\documentclass[11pt]{amsart}
\usepackage{amsfonts, amstext, amsmath, amsthm, amscd, amssymb}
\usepackage[latin1]{inputenc}
\usepackage[T1]{fontenc}
\usepackage{layout}

\usepackage[english,german]{babel}
\usepackage{enumerate}

\usepackage{type1cm}% Uses type 1 version of Computer Modern fonts

\usepackage[text={5in,7.5in}]{geometry}

\renewcommand{\leq}{\leqslant}
\renewcommand{\geq}{\geqslant}
\newcommand{\N}{\mathbb{N}}
\newcommand{\Z}{\mathbb{Z}}

\newcommand{\R}{\mathbb{R}}
\newcommand{\C}{\mathbb{C}}
\newcommand{\End}{\mathrm{End}}

\newcommand{\Tr}{\mathop{\mathrm{Tr}}}
\newcommand{\STr}{\mathop{\mathbf{S}\mathrm{Tr}}}

\newcommand{\LIM}{\mathop{\mathrm{LIM}}}

\newtheorem{thm}{Theorem}
\newtheorem{lem}{Lemma}
\newtheorem{prop}{Proposition}
\newtheorem{cor}{Corollary}

 \title[Anomaly formulas for the analytic torsion on bordisms]{Anomaly formulas for the complex-valued analytic torsion on compact bordisms}
  \author{Osmar MALDONADO MOLINA}
  \date{\today}
  \thanks{The author was supported by the IK I008-N from the University of Vienna and
  the grant P19392-N13 from the Austrian Science Fund (FWF)}
  \address{Department of Mathematics,
  University of Vienna,
  Nordbergstrasse 15,
  A-1090 Vienna,
  Austria}
 \email{osmar.maldonado@univie.ac.at}

\begin{document}
\selectlanguage{english}
\thispagestyle{empty}
\begin{abstract}
We extend the complex-valued analytic torsion, introduced by Burghelea and Haller on closed manifolds,
to compact Riemannian bordisms. We do so by considering a flat complex vector bundle over a compact Riemannian manifold, endowed with
a fiberwise nondegenerate symmetric bilinear form. The Riemmanian metric and the bilinear form are used to define
non-selfadjoint Laplacians acting on vector-valued smooth forms under absolute and relative boundary conditions. In the process to define the complex-valued analytic torsion,
we study spectral properties
associated to these generalized Laplacians.
As main results, we obtain anomaly formulas for the complex-valued analytic torsion.
Our reasoning takes into account that the coefficients in the heat trace asymptotic expansion associated to
the boundary value problem under consideration, are locally computable.
The anomaly formulas for the complex-valued Ray--Singer torsion are obtained
by using the corresponding ones for the Ray--Singer metric, obtained by Br\"uning and Ma on manifolds with boundary, and
an argument of analytic continuation. In odd dimensions, our anomaly formulas are in accord with the corresponding results of Su, without requiring
the variations of the Riemannian metric and bilinear structures to be supported in the interior of the manifold.
\end{abstract}

\maketitle
\section*{Introduction}
In this paper, we denote by $(M,\partial_{+}M,\partial_{-}M)$ a compact Riemannian bordism.
That is, $M$ is a compact Riemannian manifold of dimension $m$, with Riemannian metric $g$, whose boundary $\partial M$ is
the disjoint union of two closed submanifolds $\partial_{ +}M$ and $\partial_{-}M$. For $E$ a
flat complex vector bundle over $M$, we consider generalized Laplacians acting on the space $\Omega(M;E)$ of $E$-valued smooth differential forms on $M$
satisfying absolute boundary conditions on $\partial_{+}M$ and relative boundary conditions on $\partial_{-}M$.

We study the \textit{complex-valued} Ray--Singer torsion
on $(M,\partial_{+}M,\partial_{-}M)$. This torsion was
introduced by Burghelea and Haller on closed manifolds,
see \cite{Burghelea-Haller} and \cite{Burghelea-Haller2}, as a complex-valued version for
the real-valued Ray--Singer torsion, originally studied by Ray and Singer in \cite{Ray-Singer}
for unitary flat vector bundles on closed manifolds.
Our main results are
Theorem \ref{Theorem infinitesimal_constant_terms_in_asymptotic_expansions_bilinear_case_relative_and absolute_boundary_conditions}
 and Theorem
\ref{Theorem_Anomaly_formulas_complex_Analytic_Ray_Singer_torsion}.
In Theorem
\ref{Theorem_Anomaly_formulas_complex_Analytic_Ray_Singer_torsion}, we provide so-called \textit{anomaly formulas} providing a logarithmic derivative for
the complex-valued analytic torsion on compact Riemannian bordisms and its proof is based on the work by
Br\"uning and Ma in \cite{Bruening-Ma}
for the real-valued Ray--Singer torsion on manifolds with boundary.

The classical (\textit{real-valued}) Ray--Singer analytic torsion, see \cite{Ray-Singer}, \cite{Lueck}, \cite{Cheeger79}, \cite{Mueller78} and others,
is defined in terms of a selfadjoint Laplacian $\Delta_{E,g,h}$, constructed by using
a Hermitian metric on the bundle, the Riemannian metric $g$ and a flat connection $\nabla^{E}$ on $E$. In this paper $\Delta_{E,g,h}$ is referred as the \emph{Hermitian Laplacian}.
In \cite{Bismut-Zhang}, Bismut and Zhang interpreted the analytic torsion as a Hermitian metric in certain determinant line, and called it the Ray--Singer metric, see also \cite{Bruening-Ma2}.
In this paper, we also adopt this approach.
The Ray--Singer metric on manifolds with boundary has been intensively studied by several authors, among them
\cite{Ray-Singer}, \cite{Cheeger79}, \cite{Mueller78}, \cite{Mueller93}, \cite{Lueck}, \cite{Dai-Fang}, \cite{Bruening-Ma} and \cite{Bruening-Ma2}. In particular, we are interested in the work of Br\"uning and Ma in \cite{Bruening-Ma}, where the variation of the Ray--Singer metric,
with respect to smooth variations
on the underlying Riemannian and Hermitian metrics, was computed.

In order to define the \textit{complex-valued} Ray--Singer torsion, we
assume $E$ admits a fiberwise nondegenerate symmetric bilinear form $b$ and we proceed as in \cite{Burghelea-Haller}. The
bilinear form $b$ and the Riemannian
metric $g$ induce a nondegenerate symmetric bilinear form on $\Omega(M;E)$ which is denoted by $\beta_{g,b}$.
With this data, one constructs generalized Laplacians
$\Delta_{E,g,b}:\Omega(M;E)\rightarrow\Omega(M;E)$, also referred as \textit{bilinear Laplacians}. These generalized Laplacians are formally symmetric,
with respect to $\beta_{g,b}$ on the space of smooth forms satisfying
the boundary conditions specified above.

In Section
\ref{Section_Complex valued analytic torsion on Manifolds with Boundary}, we use known
theory on boundary value problems for differential operators
to treat ellipticity, regularity and spectral properties for $\Delta_{E,g,b}$. In particular,
under the specified
elliptic boundary conditions, $\Delta_{E,g,b}$ extends
to a not necessarily selfadjoint closed unbounded operator in the $L^{2}$-norm, it
has compact resolvent and discrete spectrum,
all its eigenvalues are of finite multiplicity, its (generalized) eigenspaces
contain smooth differential forms only and the restriction of $\beta_{g,b}$ to each of these
is also a nondegenerate bilinear form. Proposition
\ref{Proposition_from_Hodge_decomposition_2} gives Hodge decomposition results
in this setting, which are analog to the Hermitian situation, described for instance in
\cite{Cheeger79}, \cite{Mueller78}, \cite{Lueck} and more recently
in \cite{Bruening-Ma2}. Section \ref{Section_Complex valued analytic torsion on Manifolds with Boundary} ends with
Proposition \ref{Proposition_Absolute_relative_cohomology} stating
that the $0$-generalized eigenspace of $\Delta_{E,g,b}$ \emph{still}
computes relative cohomology $H(M,\partial_{-}M;E)$, without necessarily being isomorphic to it.

In Section \ref{section_Heat_asymptotics and anomaly formulas}, we recall generalities on the
coefficients of the heat kernel asymptotic expansion
for an elliptic boundary value problem.
These coefficients are spectral invariants and
locally computable as polynomial functions in the jets of the symbols of
the operators under consideration, see \cite{Greiner}, \cite{Seeley}, \cite{Seeleya} and \cite{Seeleyb}.
%By Weyl's first Theorem of invariant theory,
%these coefficients are expressed as universal polynomial in terms of locally computable geometric invariants;
This fact %studied in detail by Gilkey in \cite{Gilkey1} and \cite{Gilkey2},
provides
the key ingredient
in the proofs
of Theorem
\ref{Theorem infinitesimal_constant_terms_in_asymptotic_expansions_bilinear_case_relative_and absolute_boundary_conditions},
leading to Theorem \ref{Theorem_Anomaly_formulas_complex_Analytic_Ray_Singer_torsion}. In \cite{Bruening-Ma}, 
based on the computation of the coefficients of the constant terms in the heat trace asymptotic expansion for the
Hermitian Laplacian under absolute boundary conditions, Br\"uning and Ma obtained anomaly formulas for the Ray--Singer metric.
First, we use Poincar\'e duality in terms of
Lemma \ref{Lemma_relation_between the infinitesimal_constant_terms_in_asymptotic_expansions_Hermitian_relative_and absolute_boundary_conditions_total},
to infer from  \cite{Bruening-Ma},
the corresponding coefficients for the Hermitian Laplacian under relative boundary conditions and then we derive 
those corresponding to Hermitian Laplacian on the bordism $(M,\partial_{+}M,\partial_{-}M)$ under absolute and relative boundary conditions, see Proposition
\ref{Proposition_infinitesimal_constant_term_in_asymptotic_expansion_Hermitian_case_relative_boundary_conditions} and
Theorem \ref{Theorem infinitesimal_constant_terms_in_asymptotic_expansions_Hermitian_case_relative_and absolute_boundary_conditions}. 
We point out here that the anomaly formulas for the Ray--Singer metric
in Theorem
\ref{Theorem infinitesimal_constant_terms_in_asymptotic_expansions_Hermitian_case_relative_and absolute_boundary_conditions}
were also obtained by Br\"uning and Ma in
\cite{Bruening-Ma2} continuing their work in \cite{Bruening-Ma}. Next,
in Lemma
\ref{Lemma_holomorphic_dependance_of_a_m_on_a_parameter_u}, we point out the holomorphic dependance of these coefficients
on a complex parameter. Finally, an analytic continuation argument
allows one to deduce the infinitesimal variation of these quantities
for the bilinear Laplacian on the bordism $(M,\partial_{+}M,\partial_{-}M)$
from those corresponding to the Hermitian one, see Theorem
\ref{Theorem infinitesimal_constant_terms_in_asymptotic_expansions_bilinear_case_relative_and absolute_boundary_conditions}.

In Section \ref{Subsection_Complex analytic Ray--Singer torsion}, we use the results from Section \ref{Section_Complex valued analytic torsion on Manifolds with Boundary}
and Section \ref{section_Heat_asymptotics and anomaly formulas} to
define the complex-valued analytic
torsion on a compact Riemannian bordism. Following the approach in \cite{Burghelea-Haller}, we
obtain
a nondegenerate bilinear form on the determinant line
$\det(H(M,\partial_{-} M;E))$, denoted by $\tau_{E,g,b}(0)$ and induced  by
the restriction of $\beta_{g,b}$
to the generalized $0$-eigenspace of $\Delta_{E,g,b}$.
The (inverse square of) the complex-valued Ray--Singer torsion for manifolds with boundary is
$$
\tau^{\mathsf{RS}}_{E,g,b}
:=\tau_{E,g,b}(0)\cdot\prod_{p}\left(\left.\det\right.^{\prime}
\left(\Delta_{E,g,b,p}\right)\right)^{(-1)^{p}p},
$$
where the product above is, in this situation, a non zero complex number with
$\left.\det\right.^{\prime}
\left(\Delta_{E,g,b,p}\right)$ being the
$\zeta$-regularized product of all non-zero eigenvalues of
$\Delta_{E,g,b,p}$.
For closed manifolds,
the variation of the complex analytic Ray--Singer torsion, with respect to smooth changes on the metric $g$ and the
bilinear form $b$, has been obtained in \cite[Sections 7 and 8]{Burghelea-Haller}.
% They did so,
%by computing the leading and subleading terms in the asymptotic expansion of the
%heat kernels associated with a certain class of Dirac operators,  in \cite{Burghelea-Haller}.
Burghelea and Haller obtained in \cite[Theorem 4.2]{Burghelea-Haller}
a geometric invariant by introducing appropriate correction terms.
In \cite{Su1}, by using techniques from \cite{Su-Zhang08}, \cite{Vertman09}, \cite{Cheeger79} and \cite{Mueller78},
Su generalized the complex-valued analytic Ray--Singer torsion to the situation
in which $\partial_{+}M\not=\emptyset$ (or $\partial_{-}M\not=\emptyset $). Also in \cite{Su1}, Su
proved that in odd dimensions,
the complex-valued analytic torsion does depend neither on smooth variations of the Riemannian metric nor on smooth variations of
the bilinear form,
as long as these are compactly supported in the interior of $M$. This section ends with
Theorem
\ref{Theorem_Anomaly_formulas_complex_Analytic_Ray_Singer_torsion}, which gives formulas for the
variation of the complex-valued analytic Ray--Singer torsion with respect to smooth variations
of the metric and the
bilinear form.
In analogy with the results in \cite{Burghelea-Haller},
the
anomaly formulas for the complex-valued Ray--Singer torsion are obtained by using the results for the coefficients of the constant term in the heat trace asymptotic expansion
for the bilinear Laplacian obtained in Section
\ref{section_Heat_asymptotics and anomaly formulas}.

In the Appendix, see Section \ref{section_Appendix}, for the 
reader's convenience, we recall some formalism leading to the characteristic forms
appearing in the anomaly formulas stated in Proposition \ref{Proposition_infinitesimal_constant_term_in_asymptotic_expansion_Hermitian_absolute_boundary_conditions},
Proposition \ref{Proposition_infinitesimal_constant_term_in_asymptotic_expansion_Hermitian_case_relative_boundary_conditions}, Theorem 
\ref{Theorem infinitesimal_constant_terms_in_asymptotic_expansions_Hermitian_case_relative_and absolute_boundary_conditions}, 
Theorem \ref{Theorem infinitesimal_constant_terms_in_asymptotic_expansions_bilinear_case_relative_and absolute_boundary_conditions} and 
Theorem \ref{Theorem_Anomaly_formulas_complex_Analytic_Ray_Singer_torsion}.

The anomaly formulas given in
Theorem \ref{Theorem_Anomaly_formulas_complex_Analytic_Ray_Singer_torsion} generalize
the ones obtained by Burghelea and Haller in
the closed situation in \cite{Burghelea-Haller},
and also the ones in \cite{Su1} by Su in odd dimensions:
they do not longer require $g$ and $b$ to be constant in a neighborhood of the boundary and
both kind of boundary conditions are considered at the same time. 
\subsection*{Ackowledgements}
This paper has been written as part of a PhD thesis at the university of Vienna.
I am deeply grateful to my supervisor
Stefan Haller for useful discussions, his comments and important remarks on this work.

\section{Bilinear Laplacians and Hodge decomposition on bordisms}
\label{Section_Complex valued analytic torsion on Manifolds with Boundary}
 \subsection{Some background and notation}
 Let $(M,\partial_{+}M,\partial_{-}M)$ be a compact Riemannian bordism of dimension $m$. More precisely,
$M$ is a compact connected not necessarily orientable smooth manifold of dimension $m$ with Riemannian metric $g$,
whose boundary $\partial M$ is the disjoint union of two closed submanifolds, $\partial_{+}M$ and $\partial_{-}M$, and it inherits
the Riemannian metric from $M$. We do not require the metric to satisfy any condition near the boundary.
We denote by $TM$ and $T^{*}M$ (resp. $T\partial M$ and $T^{*}\partial M$) the tangent and cotangent bundle of $M$ (resp. $\partial M$) respectively.
 We denote by  $\varsigma_{\mathsf{in}}$  the \textit{geodesic unit inwards pointing normal vector field} on the boundary.
 Let $\Theta_{M}$ (resp. $\Theta_{\partial M}$) be the orientation bundle of $TM$ (resp. $T \partial M$),
 considered as the flat real line bundle $\det (T^{*}M)\rightarrow M$ (resp. $\det (T^{*}\partial M)\rightarrow \partial M$)
 with transition functions $\{\pm 1\}$, endowed with the unique flat connection specified by the
 de-Rham differential on (twisted) forms, see \cite[page 88]{Bott-Tu}.
 For the canonical embedding $i:\partial M\hookrightarrow M$, we write $\Theta_{M}|_{\partial M}:=i^{*}\Theta_{M}$ and,
 %the restriction of $\Theta_{M}$ to $\partial M$
 as real line bundles over $\partial M$,  $\Theta_{M}|_{\partial M}$
 and $\Theta_{\partial M}$ are identified as follows: over the boundary, a section $\beta$ of $\det(T^{*}\partial M)$ is identified
 with the section $-\varsigma^{\mathsf{in}}\wedge\beta$ of $\det(T^{*}M)|_{\partial M}$, where $\varsigma^{\mathsf{in}}:=g(\cdot,\varsigma_{\mathsf{in}})$ is 
 the 1-form dual to $\varsigma_{\mathsf{in}}$.
 For $T M$ and $T \partial M$, the corresponding
 Levi--Civit\`a connections are denoted by $\nabla$ and by $\nabla^{\partial}$ respectively.
 Recall the Hodge $\star$-operator
 $
 \star_{q}:=\star_{g,q}:\Omega^{q}(M)\rightarrow\Omega^{m-q}(M;\Theta_{M}),
 $
 i.e., the linear isomorphism
 defined by
 $
 \alpha\wedge\star\alpha^{\prime}=\langle\alpha,\alpha^{\prime}\rangle_{g}\mathsf{vol}_{g}(M),
 $
 for $\alpha,\alpha^{\prime}\in\Omega^{q}(M)$ and $0 \leq q\leq m$, where $\mathsf{vol}_{g}(M)\in\Omega^{m}(M;\Theta_{M})$ is the volume form of $M$.

 In this paper, we consider a flat complex vector bundle $E$ over $M$, with a flat connection $\nabla^{E}$, and denote by $\Omega(M;E)$ be the space
 of $E$-valued smooth differential forms on $M$, endowed with the de-Rahm differential $\mathop{d_{E}}:=d_{\nabla^{E}}$.
 Moreover,
 assume $E$ is endowed with a fiber-wise nondegenerate symmetric bilinear form $b$. We denote by $E^{\prime}$
 the flat complex vector bundle dual to $E$ with the induced flat connection $\nabla^{E^{\prime}}$ and bilinear form $b^{\prime}$
 dual to $\nabla^{E}$ and $b$ respectively. Recall that one is always able to fix a (positive definite) Hermitian structure on $E$ (in
 Section \ref{Section_Involutions, bilinear and Hermitian forms}, we choose for instance
 a Hermitian structure compatible with the nondegenerate symmetric bilinear form).
 By choosing a Hermitian structure on $E$ and using the Riemannian metric on $M$,
 consider the induced $L^{2}$-norm on $\Omega(M;E)$ and denote by $L^{2}(M;E)$ its $L^{2}$-completion.
 Recall that $L^{2}(M;E)$ is independent the chosen Hermitian and Riemannian structures.

\subsection{Generalized Laplacians on compact bordisms}
\label{subsection_Laplacians and Boundary conditions}
As a first step to define the complex-valued analytic torsion on a compact bordism,
we recall certain generalized Laplacians which were introduced in \cite{Burghelea-Haller} on closed manifolds.
%--this is possible if and only if $E$ is the complexification of a real vector bundle.
The nondegenerate symmetric bilinear form $b$ on $E$ and the Riemannian metric $g$ on $M$ permit to define a nondegenerate symmetric bilinear form on
$\Omega(M;E)$ by
$$
\beta_{g,b}(v,w):=\int_{M}\Tr(v\wedge \star_{b}w)
$$ where
$\Tr:\Omega(M,E\otimes E^{\prime}\otimes\Theta_{M})\rightarrow\Omega(M;\Theta_{M})$ is the trace map,
induced by the canonical pairing between the bundles $E$ and $E^{\prime}$, and
the map
$$
\star_{b,q}:=\star_{q}\otimes b:\Omega^{q}(M;E)\rightarrow\Omega^{m-q}(M;E^{\prime}\otimes\Theta_{M})
$$
is defined by using the Hodge $\star$-operator
$\star_{q}$ and the isomorphism of vector bundles between $E$ and $E^{\prime}$,
specified by the bilinear form $b$, also denoted by the same symbol. Thus, one defines
$d^{\sharp}_{E,g,b,q}:\Omega^{q}(M;E)\rightarrow\Omega^{q-1}(M;E)$ by
\begin{equation}
 \label{formula_definition_d_transpose}
 d^{\sharp}_{E,g,b,q}:=
 (-1)^{q}{\star_{b,q-1}}^{-1}d_{E^{\prime}\otimes\Theta_{M},m-q}\star_{b,q},
\end{equation}
where ${\star_{b,q-1}}^{-1}$ is the inverse of $\star_{b,q-1}$ and
$d_{E^{\prime}\otimes\Theta_{M}}$ is the de-Rham differential on $\Omega(M;E^{\prime}\otimes\Theta_{M})$
induced by the dual connection on $E^{\prime}$. It can easily be checked that $d^{\sharp}_{E,g,b}$ is a codifferential on $\Omega(M;E)$. In this way, the operator
\begin{equation}
 \label{formula_definition_Laplacian}
 \Delta_{E,g,b,q}:=d_{E,q-1}d^{\sharp}_{E,g,b,q}+d^{\sharp}_{E,g,b,q+1}d_{E,q}:\Omega^{q}(M;E)\rightarrow \Omega^{q}(M;E),
\end{equation}
is an operator of Laplace type, or generalized Laplacian in the sense that its principal symbol is a scalar positive real number, i.e, $\Delta_{E,g,b}$ is elliptic. For simplicity,
the operator $\Delta_{E,g,b}$ in (\ref{formula_definition_Laplacian}) will be called the \emph{bilinear Laplacian}.
A straightforward use of Stokes' Theorem leads to the Green's formulas:
 \begin{eqnarray}
 %\begin{array}{l}
 \scriptstyle\beta_{g,b}(\mathop{d_{E}}v,w)-
 \beta_{g,b}(v,\mathop{d^{\sharp}_{E,g,b}}w)&=&\scriptstyle\int_{\partial M}i^{*}(\Tr(v\wedge \star_{b} w)),\nonumber\\
 \label{Lemma_Stokes_theorem_transpose_Greens_Formulas_non_self_adjoint_Laplacian}
 \scriptstyle\beta_{g,b}(\Delta_E v,w)-\beta_{g,b}(v,\Delta_E w)
            &=&\scriptstyle\int_{\partial M} i^{*}(\Tr(\mathop{d^{\sharp}_{E,g,b}} v\wedge\star_{b}w))
                -\int_{\partial M} i^{*}(\Tr(w\wedge\star_{b}\mathop{d_{E}}v))\\	
            &&\nonumber\scriptstyle-\int_{\partial M} i^{*}(\Tr(\mathop{d^{\sharp}_{E,g,b}} w\wedge\star_{b}v))
            +\int_{\partial M} i^{*}(\Tr(v\wedge\star_{b}\mathop{d_{E}} w)).
 %\end{array}
 \end{eqnarray}
 for $v,w\in \Omega(M;E)$.

\subsection{Boundary conditions}
\label{section_Absolute_and_relative_boundary_conditions_on_bordisms}
In order to study analytic and spectral properties of $\Delta_{E,g,b}$, we
impose elliptic boundary conditions. We denote by $i_{\pm}:\partial_{\pm}M\hookrightarrow M$ the canonical embedding of
$\partial_{\pm}M$ into $M$ respectively. For a form $w\in\Omega(M;E)$, we say that
$w$ satisfies \textit{relative
boundary conditions} on $\partial_{-}M$ if
$
i^{*}_{-}w=0$ and $i^{*}_{-}\mathop{d^{\sharp}_{E,g,b}}w=0$
and $w$ satisfies \textit{absolute boundary conditions} on $\partial_{+}M$ if
$i^{*}_{+}\star_{b}w=0$ and $i^{*}_{+}d^{\sharp}_{E^{\prime}\otimes\Theta_{M},g,b}\star_{b} w=0.$
The space of smooth forms  satisfying relative
boundary conditions on $\partial_{-}M$ \textit{and}
absolute boundary conditions on $\partial_{+}M$ is
\begin{equation}
\label{Differential_forms_satisfying_boundary_conditions_original}
\scriptstyle\Omega(M;E)|_{\mathcal{B}}:=
\left\{ w\in\Omega(M;E)\left|\begin{array}{cc}
    \scriptstyle i^{*}_{+}\star_{b} w=0, & \scriptstyle i^{*}_{-}w=0\\ \scriptstyle
    i^{*}_{+}d^{\sharp}_{E^{\prime}\otimes\Theta_{M},g,b}\star_{b} w=0, & \scriptstyle i^{*}_{-}\mathop{d^{\sharp}_{E,g,b}}w=0
    \end{array}\right\}.\right.
\end{equation}
For simplicity, a form satisfying boundary conditions in (\ref{Differential_forms_satisfying_boundary_conditions_original}) will be referred as satisfying
\textit{absolute/relative boundary conditions on $(M,\partial_{+}M,\partial_{-}M)$}.
The integrants on the right of
formulas in (\ref{Lemma_Stokes_theorem_transpose_Greens_Formulas_non_self_adjoint_Laplacian})
vanish, on forms in $\Omega(M;E)|_{\mathcal{B}}$.
%The operator $\Delta_{E,g,b}$ with domain of definition $\Omega(M;E)|_{\mathcal{B}}$
%does not
%extend to a selfadjoint operator on $L^{2}(M;E)$.
The boundary conditions in (\ref{Differential_forms_satisfying_boundary_conditions_original}) are an example of \textit{mixed boundary conditions}, which provide elliptic boundary conditions
for operators of Laplace type, see \cite{Gilkey2}.% As in the closed situation, we will see that the
%spectrum of the bilinear Laplacian possesses properties close to
%those of a Hermitian Laplacian, see
%Section \ref{section_The spectrum of the bilinear Laplacian on manifolds with boundary}.

Now we describe boundary operators implementing the boundary conditions in (\ref{Differential_forms_satisfying_boundary_conditions_original}).
Consider
$E_{\pm}:=i^{*}_{\pm }E$ and for
$1\leq q\leq m$ define
\begin{equation}
\label{definition_Boundary_operator_on_smooth forms}
\begin{array}{lrcl}
\mathcal{B}_{E,g,b}:\phantom{a}&\Omega^{q}(M;E)&\longrightarrow&\Omega^{q-1}(\partial_{+}M;E_{+})
\oplus\Omega^{q}(\partial_{+}M;E_{+})\\
&&&\phantom{\Omega^{m-q}}\oplus\phantom{s}
\Omega^{q}(\partial_{-}M;E_{-})\oplus\Omega^{q-1}(\partial_{-}M;E_{-})\\
&w&\mapsto &({\mathcal{B}_{+}}w,
  {\mathcal{B}_{-}}w),
\end{array}
\end{equation}
where the operators
  \begin{equation}
  \label{definition_relative_Boundary_operator_on_smooth forms_0}
\begin{array}{c}
\begin{array}{lrcl}
 {\mathcal{B}_{-}}:&\Omega^{q}(M;E)&\longrightarrow&\Omega^{q}(\partial_{-}M;E_{-})\oplus\Omega^{q-1}(\partial_{-}M;E_{-})\\
		   &w&\mapsto&(\mathcal{B}_{-}^{0}w,\mathcal{B}_{-}^{1}w)\\
\end{array}
\\ \\
\begin{array}{lrcl}
 {\mathcal{B}_{+}}_{}:&\Omega^{q}(M;E)&\longrightarrow&\Omega^{q-1}(\partial_{+}M;E_{+})\oplus\Omega^{q}(\partial_{+}M;E_{+})\\
		   &w&\mapsto&(\mathcal{B}_{+}^{0}w,\mathcal{B}_{+}^{1}w)\\
\end{array}
\end{array}
\end{equation}
are respectively defined in terms of
  \begin{equation}
  \label{definition_relative_Boundary_operator_on_smooth forms}
  \begin{array}{lcl}
      \mathcal{B}_{-}^{0}w:=i^{*}_{-}w,&&
      \mathcal{B}_{-}^{1}w:=i^{*}_{-}\mathop{d^{\sharp}_{E,g,b}}w,\\&&\\
      \mathcal{B}_{+}^{0}w:=
  \left.\star_{b}^{\partial M}\right.^{-1}\left(i^{*}_{+}\star_{b}w\right),&&
      \mathcal{B}_{+}^{1}w:=\left.\star_{b}^{\partial M}\right.^{-1}\left(i^{*}_{+}d^{\sharp}_{E^{\prime}\otimes\Theta_{M},g,b^{\prime}}\star_{b}w\right).
  \end{array}
  \end{equation}
A form $w$ satisfies the boundary conditions, i.e.,  $w\in\Omega(M;E)|_{\mathcal{B}}$, if and only if
 $\mathcal{B}w=0$.

\begin{lem}
\label{Lemma_image_by_d_of_omega_2_and_symmetry_beta_boundary_conditions}
    For a subspace $\mathsf{X}\subseteq\Omega(M;E)$, denote by
$
\mathsf{X}|_{\mathfrak{B}}:=\{w\in\mathsf{X}|\mathfrak{B}w=0\}
$
the space of smooth forms in $\mathsf{X}$ which satisfy the boundary conditions specified by the vanishing of the operator
$\mathfrak{B}\in\{\mathcal{B}_{\pm}^{0},\mathcal{B}_{\pm}^{1},\mathcal{B}_{\pm},\mathcal{B}\}$. Set
\begin{equation}
 \label{Definition_subspace_space_of_smooth_forms_with_boundary_conditions}
\begin{array}{lcr}
 \mathsf{X}|_{\mathcal{B}^{0}}:=\mathsf{X}|_{\mathcal{B}_{-}^{0}}\cap\mathsf{X}|_{\mathcal{B}^{0}_{+}}.\\
 \end{array}
\end{equation}
Then the following assertions hold
    \begin{enumerate}
    \item[(a)]\label{remark_trivial_inclusions}
    $\mathsf{X}|_{\mathcal{B}}=\mathsf{X}|_{\mathcal{B}^{0}}\cap\mathsf{X}|_{\mathcal{B}_{-}^{1}}\cap\mathsf{X}|_{\mathcal{B}_{+}^{1}}$ and $\mathsf{X}|_{\mathcal{B}}
\subset \mathsf{X}|_{\mathcal{B}^{0}} \subset \mathsf{X}|_{\mathcal{B}_{-}^{0}},$
    \item[(b)]  %$\mathop{d_{E}}$ leaves invariant the space $\Omega(M;E)|_{\mathcal{B}_{-}^{0}}$:
    $\mathop{d_{E}}(\Omega(M;E)|_{\mathcal{B}_{-}^{0}})\subset\Omega(M;E)|_{\mathcal{B}_{-}^{0}},$
    \item[(c)]
	    $\mathop{d_{E}}(\Omega(M;E)|_{\mathcal{B}})\subset\Omega(M;E)|_{\mathcal{B}^{0}}$ and
	    $\mathop{d^{\sharp}_{E,g,b}}(\Omega(M;E)|_{\mathcal{B}})\subset\Omega(M;E)|_{\mathcal{B}^{0}},$
    \item[(d)]
    If $v\in \Omega(M;E)|_{\mathcal{B}_{-}^{0}}$ and $w\in \Omega(M;E)|_{\mathcal{B}}$ then
    $\beta_{g,b}(\mathop{d_{E}}v,\mathop{d^{\sharp}_{E,g,b}}w)=0,$
    \item[(e)]
    If $v,w\in \Omega(M;E)|_{\mathcal{B}^{0}}$, then
    $\beta_{g,b}(\mathop{d_{E}}v,w)=\beta_{g,b}(v,\mathop{d^{\sharp}_{E,g,b}}w)                ,$
    \item[(f)]
    If $v,w\in \Omega(M;E)|_{\mathcal{B}}$, then  $\beta_{g,b}(\Delta_{E,g,b}v,w)=\beta_{g,b}(v,\Delta_{E,g,b}w).$
    \end{enumerate}
    \end{lem}
    \begin{proof}
    The first assertion is obvious. The remaining assertions follow from
     (\ref{Definition_subspace_space_of_smooth_forms_with_boundary_conditions}),
     (\ref{Differential_forms_satisfying_boundary_conditions_original}), the Green's formulas in
     (\ref{Lemma_Stokes_theorem_transpose_Greens_Formulas_non_self_adjoint_Laplacian}) and
     straightforward manipulations coming from the definition of the operators and spaces above.
    \end{proof}

\subsection{Boundary conditions and Poincar\'e duality}
\label{Boundary_conditions_and_Poincae_duality}
 Consider the Riemannian bordism $(M,\partial_{+}M,\partial_{-}M)$. The boundary value problem specified by the operator
 $\Delta_{E,g,b}$ acting on the space $\Omega(M;E)|_{\mathcal{B}}$ as defined by
 (\ref{Differential_forms_satisfying_boundary_conditions_original}),
 will be denoted
 by
 \begin{equation}
 \label{notation_Boundary_value_problem}
 [\Delta,\mathcal{B}]^{E,g,b}_{(M,\partial_{+}M,\partial_{-}M)}.
 \end{equation}
 Let us denote by $(M,\partial_{+}M,\partial_{-}M)^{\prime}:=(M,\partial_{-}M,\partial_{+}M)$ the \textit{dual bordism} to $(M,\partial_{+}M,\partial_{-}M)$.
 Then, we are interested in $[\Delta,\mathcal{B}]^{E^{\prime}\otimes\Theta,g,b^{\prime}}_{(M,\partial_{+}M,\partial_{-}M)^{\prime}}$
 the dual boundary value problem to (\ref{notation_Boundary_value_problem}),
 corresponding to the bilinear Laplacian $\Delta_{E^{\prime},g,b^{\prime}}$ acting on $E^{\prime}\otimes\Theta_{M}$-valued forms (where
 the flat complex vector bundle $E^{\prime}$ is endowed with the dual connection $\nabla^{E^{\prime}}$ and
 dual bilinear form $b^{\prime}$) under the boundary conditions specified by
 the vanishing of the boundary operator $\mathcal{B}^{\prime}$, i.e., the same operator from (\ref{definition_Boundary_operator_on_smooth forms}) but associated to
 $(M,\partial_{+}M,\partial_{-}M)^{\prime}$.
 The boundary value problem in (\ref{notation_Boundary_value_problem}) is naturally intertwined with its dual one
 by means of the Hodge $\star$-operator.  Indeed, by the very definition of these operators, we have
 the equality $$\star_{b}\mathop{d^{\sharp}_{E,g,b}}\mathop{d_{E}}=d_{E^{\prime}\otimes\Theta_{M}}d^{\sharp}_{E^{\prime}\otimes\Theta_{M},g,b^{\prime}}\star_{b}$$ so that
 $$\star_{b}\Delta_{E,g,b}=\Delta_{E^{\prime}\otimes\Theta_{M},g,b^{\prime}}\star_{b},$$ and
 $$w\in\Omega^{q}(M;E)|_{\mathcal{B}}\Longleftrightarrow\star_{b}w\in\Omega^{m-q}(M;E^{\prime}\otimes\Theta_{M})|_{\mathcal{B}^{\prime}}.
 $$
 That is,  the Hodge-$\star_{b}$-operator intertwines the roles of $\partial_{+}M$ and $\partial_{-}M$ in
 (\ref{notation_Boundary_value_problem}) and its dual.

 As a special case,
 if $\partial_{+}M=\partial M$ and
 $\partial_{-}M=\emptyset$ (resp. $\partial_{+}M=\emptyset$ and $\partial_{-}M=\partial M$),
 then
 $[\Delta,\mathcal{B}]^{E,g,b}_{(M,\partial M,\emptyset)}$,
 (resp. $[\Delta,\mathcal{B}]^{E,g,b}_{(M,\emptyset,\partial M)}$)
 is the boundary value problem where absolute (resp. relative) boundary conditions \textit{only} are imposed on
 $\partial M$.

\subsection{Hermitian boundary value problems}
We recall some facts for the Hermitian situation.
By using a Hermitian structure $h$ on $E$,
instead of the bilinear form $b$, all over in the considerations
above, one has
$\ll v,w\gg_{g,h}:=\int_{M}\Tr(v\wedge \star_{h}w)$ a Hermitian product on $\Omega(M;E)$, where $\star_{h}$ is in this case a fiber-wise complex anti-linear isomorphism induced
by $h$ and $\star_{g}$. Then, associated to this data, one considers  a differential $d_{E}$, a codifferential $d^{*}_{E,g,h}$ and a Laplacian
$$\Delta_{E,g,h}:=d_{E}d^{*}_{E,g,h}+d_{E,g,h}^{*}d_{E}:\Omega(M;E)\rightarrow\Omega(M;E),$$ which is formally selfadjoint with respect to  $\ll v,w\gg_{g,h}$,
under absolute/relative boundary conditions on $(M,\partial_{+}M,\partial_{-}M)$.
Let
$\Omega(M;E)|^{h}_{\mathcal{B}}$ be the space of $E$-valued smooth forms satisfying absolute/relative boundary conditions on $(M,\partial_{+}M,\partial_{-}M)$ defined as in
(\ref{Differential_forms_satisfying_boundary_conditions_original}) but using instead the Hermitian form $h$. In order to
distinguish this problem from the bilinear one, we refer to it as the \emph{Hermitian boundary value problem}.

The Hermitian boundary value problem
is an elliptic boundary value problem, see \cite{Gilkey1} and \cite{Gilkey2}. This permits one to consider
$\Delta_{E,g,h}$, as an unbounded operator in the $L^{2}$-norm and extend it
%$$\Delta_{E,g,h}:\Omega(M;E)|^{h}_{\mathcal{B}}\subset L^{2}(M;E)\rightarrow\subset L^{2}(M;E)$$
to a selfadjoint operator with domain of
definition being the $H_{2}$-Sobolev closure
of $\Omega(M;E)|^{h}_{\mathcal{B}}$; see \cite{Lueck}, \cite{Cheeger79}, \cite{Mueller78}, \cite{Gilkey1} and \cite{Gilkey2}.
In particular, in this Hermitian setting, there are well-known Hodge-decomposition results.
For instance, if $\mathcal{H}^{q}_{\Delta_{\mathcal{B}}}(M;E)$ is the space $\ker\left(\Delta_{E,g,h}\right)\cap\Omega^{q}(M;E)|^{h}_{\mathcal{B}}$
of $q$-\textit{Harmonic forms} satisfying boundary conditions, then
\cite[Theorem 1.10]{Lueck} (see also \cite[page 239]{Mueller78}) states that for each $v\in \Omega^{q}(M;E)|^{h}_{\mathcal{B}^{0}}$, there exist unique
$v_{0}\in\mathcal{H}^{q}_{\Delta_{\mathcal{B}}}(M;E)$, $
v_{1}\in \mathop{d_{E}}(\Omega^{q-1}(M;E)|^{h}_{\mathcal{B}^{0}})$ and $v_{2}\in d^{*}_{E,g,h}(\Omega^{q+1}(M;E)|^{h}_{\mathcal{B}^{0}})
$ such that
$
v=v_{0}+v_{1}+v_{2},
$
where we have used the notation suggested in (\ref{Definition_subspace_space_of_smooth_forms_with_boundary_conditions}) associated to $h$.
Moreover, the Hodge--De-Rham tells us that relative cohomology exactly coincides with the space of Harmonic forms of
the Hermitian Laplacian:
\begin{equation}
\label{equation_Hodge_Hermitian}
 \mathcal{H}^{q}_{\Delta_{\mathcal{B}}}(M;E)\cong H^{q}(M,\partial_{-} M;E).
\end{equation}
In the bilinear seeting, the isomorphism in (\ref{equation_Hodge_Hermitian}) does no longer holds, but we have instead Proposition \ref{Proposition_Absolute_relative_cohomology} below.
One uses the isomorphism in (\ref{equation_Hodge_Hermitian}) to define the Ray--Singer metric on manifolds with boundary, as a Hermitian metric on the determinant line in (relative) cohomology. This problem
has been studied by many authors, see for instance \cite{Ray-Singer}, \cite{Lueck}, \cite{Cheeger79}, \cite{Mueller78}, \cite{Dai-Fang}, \cite{Bruening-Ma} and \cite{Bruening-Ma2}.
In particular, we are interested in the work by Br\"uning and Ma in \cite{Bruening-Ma}, where the case $\partial_{-}M=\emptyset$ was studied.

\subsection{The spectrum of the bilinear Laplacian}
\label{section_The spectrum of the bilinear Laplacian on manifolds with boundary}
%Boundary ellipticity value problems we use the notion of \textit{ellipticity
%with respect to a conic subset of} $\C$.
Consider the
boundary valued problem $[\Delta,\mathcal{B}]^{E,g,b}_{(M,\partial_{+}M,\partial_{-}M)}$.
Here we denote by $H_{s}(M;E)$ for $s\geq 0$, the corresponding Sobolev completions of $\Omega(M;E)$ with respect to a Hermitian metric on $E$.
By \cite[Section 20.1]{Hoermander} and \cite[Chapter 1]{Agranovich},
the operators $\Delta_{E,g,b}$ and $\mathcal{B}^{i}_{E,g,b}$ extend as a linear \textit{bounded} operators %$\Delta_{E,g,b}:H_{2}(M;E)\rightarrow L^{2}(M;E)$.
  \begin{equation}
   \label{equation_D_extends_bounded_on_H_s}
   \Delta_{E,g,b}:H_{2}(M;E)\rightarrow L^{2}(M;E)
  \end{equation}
  and
  \begin{equation}
   \label{equation_B_extends_bounded_on_H_s}
  \mathcal{B}^{i}_{E,g,b}:H_{2}(M,E)\rightarrow  H_{\frac{1}{2}}(\partial M;E|_{\partial M})\oplus H_{\frac{3}{2}-i}(\partial M,E|_{\partial M})
 \end{equation}
respectively and again these are independent on the chosen Hermitian structure.

By the
\textit{$L^{2}$-realization} of the bilinear Laplacian is understood the same operator in (\ref{equation_D_extends_bounded_on_H_s}) but considered as the
\textit{unbounded} operator in $L^{2}(M;E)$
\begin{equation}
\label{Unbounded_Laplacian_boundary_operator}
 \Delta_{\mathcal{B}}: \mathcal{D}(\Delta_{\mathcal{B}})\subset L^{2}(M;E) \rightarrow  L^{2}(M;E)
\end{equation}
 with domain of definition
 \begin{equation}
 \label{Unbounded_Laplacian_boundary_operator_DOMAIN}
 \mathcal{D}(\Delta_{\mathcal{B}}):= \overline{\Omega(M;E)|_{\mathcal{B}}}^{H_{2}}.
 \end{equation}
The boundary value problem $[\Delta,\mathcal{B}]^{E,g,b}_{(M,\partial_{+}M,\partial_{-}M)}$
is elliptic with respect to the cone $\C\backslash (0,\infty)$,  see \cite[Lemma 1.5.3]{Gilkey2}. Boundary ellipticity
guarantees the existence of elliptic estimates, see
\cite[Theorem 6.3.1]{Agranovich}  and \cite[Theorem 20.1.2]{Hoermander}. Then, elliptic estimates permit one to conclude that the
$L^{2}$-realization of the bilinear Laplacian
is a \textit{closed} unbounded operator in $L^{2}(M;E)$, which coincides with
the $L^{2}$-closure extension of
$$
\Delta_{E,g,b}:\Omega(M;E)|_{\mathcal{B}}\subset L^{2}(M;E)\rightarrow\Omega(M;E)\subset L^{2}(M;E),
$$
regarded as unbounded operator in $L^{2}(M;E)$.
\begin{lem}
 \label{Proposition_spectrum_Delta_L_2_invariance_of_domain_of_delta_contained_in_domain_of_d}
  Let $\Delta_{\mathcal{B}}$ be the unbounded operator with domain of definition $\mathcal{D}(\Delta_{\mathcal{B}})$ given in
  (\ref{Unbounded_Laplacian_boundary_operator_DOMAIN}).
  This operator is densely defined in $L^{2}(M;E)$, possesses a non-empty resolvent set,
  its resolvent is compact and its spectrum is discrete.
  More precisely, for every $\theta>0$, there exists $R>0$ such that $\mathbb{B}_{R}(0)$,
  the
  closed ball in $\C$ centered at 0 and radius $R$, contains at most a finite subset of
  $\mathsf{Spec}(\Delta_{\mathcal{B}})$ and the remaining part of the
  spectrum is entirely contained in the sector
  $$\Lambda_{R,\theta}:=\{\left. z\in\C\right|-\theta < \arg(z)
 <\theta\text{ and }|z|\geq R\}.$$
 Furthermore, for every $\lambda\not\in\Lambda_{R,\theta}$ large enough, there is $C>0$, for which
 $
 \|(\Delta_{\mathcal{B}}-\lambda)^{-1}\|_{L^{2}}\leq C/|\lambda|.
 $
\end{lem}
\begin{proof}
This follows from boundary ellipticity with respect to the conical set $\C\backslash (0,\infty)$. For a detailed discussion on this result (which holds also
in the more general setting of pseudodifferential boundary value problems for operators), we refer the reader to \cite[Theorem 3.3.2, Corollary 3.3.3 and Remark 3.3.4]{Grubb}
(see also \cite[Section 1.5]{Grubb}).
\end{proof}

\subsection{Generalized eigenspaces}
    By Lemma \ref{Proposition_spectrum_Delta_L_2_invariance_of_domain_of_delta_contained_in_domain_of_d}, $\mathsf{Spec}(\Delta_{\mathcal{B}})$ is discrete and then,
    for each $\lambda\in\mathsf{Spec}(\Delta_{\mathcal{B}})$, we choose
    $\gamma(\lambda)$ a closed counter-clock-wise oriented curve surrounding $\lambda$
    as the unique point of $\mathsf{Spec}(\Delta_{\mathcal{B}})$. Consider the corresponding Riesz or
    spectral projection:
    \begin{equation}
    \label{Riesz_Projection}
     \begin{array}{rrcl}
    \mathsf{P}_{\Delta_{\mathcal{B}}}(\lambda):
    &L^{2}(M;E)&\rightarrow&\mathcal{D}\left(\Delta_{\mathcal{B}}\right)\subset L^{2}(M;E),\\
    &w&\mapsto&-(2\pi i)^{-1}\int_{\gamma(\lambda)}(\Delta_{\mathcal{B}}-\mu)^{-1}w \mathop{d\mu}.
    \end{array}
    \end{equation}
    The integral above in (\ref{Riesz_Projection}) converges uniformly in the $L^{2}$-norm as the limit of Riemann sums, since
    the function $x\mapsto (\Delta_{\mathcal{B}}-x)^{-1}$ is analytic in a neighborhood of $\gamma(\lambda)$.
    The image of $\mathsf{P}_{\Delta_{\mathcal{B}}}(\lambda)$ in $L^{2}(M;E)$ is denoted by $$\Omega_{\Delta_{\mathcal{B}}}(M;E)(\lambda):=\mathsf{P}_{\Delta_{\mathcal{B}}}(\lambda)(L^{2}(M;E)).$$
    Since the resolvent of $\Delta_{\mathcal{B}}$ is compact, the operator $\mathsf{P}_{\Delta_{\mathcal{B}}}(\lambda)$ is
    bounded on $L^{2}(M;E)$, and  $\Omega_{\Delta_{\mathcal{B}}}(M;E)(\lambda)$ is of finite dimension,
    see \cite[Theorem 6.29]{Kato}.
    %Let us denote by
    %$$(\mathsf{Id}-\mathsf{P}_{\Delta_{\mathcal{B}}}(\lambda)):L^{2}(M;E)\rightarrow L^{2}(M;E)$$
    The image of the complementary projection to $\mathsf{P}_{\Delta_{\mathcal{B}}}(\lambda)$ on $L^{2}(M;E)$ is denoted by
    $$
    \mathsf{Im}(\mathsf{Id}-\mathsf{P}_{\Delta_{\mathcal{B}}}(\lambda)):=(\mathsf{Id}-\mathsf{P}_{\Delta_{\mathcal{B}}}(\lambda))(L^{2}(M;E)).
    $$
     Then the space $L^{2}(M;E)$ decomposes as a direct sum of Hilbert spaces compatible with the projections
     $\mathsf{P}_{\Delta_{\mathcal{B}}}(\lambda)$ and $(\mathsf{Id}-\mathsf{P}_{\Delta_{\mathcal{B}}}(\lambda))$. More precisely,  the
     following Lemma is a direct application of \cite[Theorem 6.17]{Kato}.
     \begin{lem}
     \label{Kato}
      Consider the unbounded operator $(\Delta_{\mathcal{B}},\mathcal{D}(\Delta_{\mathcal{B}}))$ from (\ref{Unbounded_Laplacian_boundary_operator}). For
      $\lambda\in \mathsf{Spec}(\Delta_{\mathcal{B}})$ consider the corresponding spectral projection
      $\mathsf{P}_{\Delta_{\mathcal{B}}}(\lambda)$. Then $\Delta_{\mathcal{B}}$ commutes with
      $\mathsf{P}_{\Delta_{\mathcal{B}}}(\lambda)$; that is,
      for $u\in\mathcal{D}(\Delta_{\mathcal{B}})$,
      we have $$\mathsf{P}_{\Delta_{\mathcal{B}}}(\lambda)u\in\mathcal{D}(\Delta_{\mathcal{B}})\text{ and }
      \mathsf{P}_{\Delta_{\mathcal{B}}}(\lambda)\Delta_{\mathcal{B}}u=\Delta_{\mathcal{B}}\mathsf{P}_{\Delta_{\mathcal{B}}}(\lambda)u.$$
      The space $L^{2}(M;E)$ decomposes as $$L^{2}(M;E)\cong\Omega_{\Delta_{\mathcal{B}}}(M;E)(\lambda)\oplus
      \mathsf{Im}(\mathsf{Id}-\mathsf{P}_{\Delta_{\mathcal{B}}}(\lambda)),$$
     such that
     $$\begin{array}{c}\mathsf{P}_{\Delta_{\mathcal{B}}}(\lambda)(\mathcal{D}(\Delta_{\mathcal{B}}))\subset\mathcal{D}(\Delta_{\mathcal{B}}),\\
     \Delta_{\mathcal{B}}(\Omega_{\Delta_{\mathcal{B}}}(M;E)(\lambda))\subset\Omega_{\Delta_{\mathcal{B}}}(M;E)(\lambda),\\
     \Delta_{\mathcal{B}}(\mathsf{Im}(\mathsf{Id}-\mathsf{P}_{\Delta_{\mathcal{B}}}(\lambda))\cap\mathcal{D}(\Delta_{\mathcal{B}}))\subset\mathsf{Im}(\mathsf{Id}-\mathsf{P}_{\Delta_{\mathcal{B}}}(\lambda)).\end{array}$$
     The operator
    \begin{equation}
    \label{Delta_restricted to an eigenspace is bounded}
    \Delta_{\mathcal{B}}|_{\Omega_{\Delta_{\mathcal{B}}}(M;E)(\lambda)}:
    \Omega_{\Delta_{\mathcal{B}}}(M;E)(\lambda)\rightarrow
    \Omega_{\Delta_{\mathcal{B}}}(M;E)(\lambda),
    \end{equation}
    is bounded on $\Omega_{\Delta_{\mathcal{B}}}(M;E)(\lambda)$,
    $\mathsf{Spec}(\Delta_{\mathcal{B}}|_{\Omega_{\Delta_{\mathcal{B}}}(M;E)(\lambda)})=\{\lambda\}$ and
    the operator
    \begin{equation}
    \label{equation_restriction_of_delta_on_complement}
    \scriptstyle(\Delta_{\mathcal{B}}-\lambda)|_{\mathcal{D}\left(\left.(\Delta_{\mathcal{B}}-\lambda)\right|_{\mathsf{Im}(\mathsf{Id}-\mathsf{P}_{\Delta_{\mathcal{B}}}(\lambda))}\right)}\phantom{:}:\phantom{:}
    \mathcal{D}\left(\left.(\Delta_{\mathcal{B}}-\lambda)\right|_{\mathsf{Im}(\mathsf{Id}-\mathsf{P}_{\Delta_{\mathcal{B}}}(\lambda))}\right)
    \rightarrow\mathsf{Im}(\mathsf{Id}-\mathsf{P}_{\Delta_{\mathcal{B}}}(\lambda)),
    \end{equation}
    with domain of definition
    $$\scriptstyle\mathcal{D}\left(\left.(\Delta_{\mathcal{B}}-\lambda)\right|_{\mathsf{Im}(\mathsf{Id}-\mathsf{P}_{\Delta_{\mathcal{B}}}(\lambda))}\right)
    \phantom{:}:=\phantom{:}\mathsf{Im}(\mathsf{Id}-\mathsf{P}_{\Delta_{\mathcal{B}}}(\lambda))\cap\mathcal{D}(\Delta_{\mathcal{B}})\subset L^{2}(M;E),$$
    is invertible, i.e.,  the spectrum of $\Delta_{\mathcal{B}}|_{\mathsf{Im}(\mathsf{Id}-\mathsf{P}_{\Delta_{\mathcal{B}}}(\lambda))}$ is
    exactly $\mathsf{Spec}\left(\Delta_{\mathcal{B}}\right)\backslash\{\lambda\}$.
    \end{lem}
   The operator $\Delta_{\mathcal{B}}|_{\Omega_{\Delta_{\mathcal{B}}}(M;E)(\lambda)}$
   in (\ref{Delta_restricted to an eigenspace is bounded})
   being bounded, its spectrum containing $\lambda$ only and
   $\Omega_{\Delta_{\mathcal{B}}}(M;E)(\lambda)$ being of finite dimension,
   the operator $(\Delta_{\mathcal{B}}-\lambda)|_{\Omega_{\Delta_{\mathcal{B}}}(M;E)(\lambda)}$ is nilpotent.

   Commutativity of $\mathsf{P}_{\Delta_{\mathcal{B}}}(\lambda)$ with
   $\Delta_{\mathcal{B}}$ on its domain $\mathcal{D}(\Delta_{\mathcal{B}})$, invariance of $\Omega_{\Delta_{\mathcal{B}}}(M;E)(\lambda)$ under $\Delta_{\mathcal{B}}$,
   and the (iterated) use of elliptic estimates with Sobolev embedding, one has
%   Moreover, since the operators $\mathsf{P}_{\Delta_{\mathcal{B}}}(\lambda)$ and
%   $\Delta_{\mathcal{B}}$ commute and $\mathcal{D}(\Delta_{\mathcal{B}})$ is invariant under
%   $\Delta_{\mathcal{B}}$, the existence of elliptic estimates, Sobolev embedding and a standard argument of induction imply
%   that if
%   $w\in\Omega_{\Delta_{\mathcal{B}}}(M;E)(\lambda)\subset\mathcal{D}(\Delta_{\mathcal{B}})\subset H_{2}(\Omega(M;E))$, then $w\in\Omega(M;E)|_{\mathcal{B}}\subset \Omega(M;E)$.
   $\Omega_{\Delta_{\mathcal{B}}}(M;E)(\lambda)\subset\Omega(M;E)|_{\mathcal{B}}\subset\Omega(M;E)$. %then $w\in\Omega(M;E)|_{\mathcal{B}}\subset \Omega(M;E)$.
   Thus each $\lambda$-eigenspace can be described as
   $$
   \scriptstyle{\Omega_{\Delta_{\mathcal{B}}}(M;E)(\lambda)
     =\left\{w\in\Omega(M;E)|_{\mathcal{B}}\left|\begin{array}{c}
                                                 \scriptstyle{\left(\Delta_{E,g,b}-\lambda\right)^{n}w\text{ }\in\text{ }\Omega(M;E)|_{\mathcal{B}},\text{ }\forall n\geq 0,}\\
						 \scriptstyle{\exists N\in\N\text{ s.t. }\left(\Delta_{E,g,b}-\lambda\right)^{n}w=0,\text{ }\forall n\geq N}
                                                \end{array}
    \right.\right\}.}
    $$\begin{lem}
    \label{Proposition_finite_dimensional_generalized_eigen_spaces}
The space $\Omega_{\Delta_{\mathcal{B}}}(M;E)(\lambda)$ is invariant under $\mathop{d_{E}}$ and $\mathop{d^{\sharp}_{E,g,b}}$.
   \end{lem}
   \begin{proof}
    We show that $\Omega_{\Delta_{\mathcal{B}}}(M;E)(\lambda)$ is invariant
    under $\mathop{d_{E}}$ and $\mathop{d^{\sharp}_{E,g,b}}$.
    Since $\Omega_{\Delta_{\mathcal{B}}}(M;E)(\lambda)$ contains  smooth differential forms only, it suffices to show that
    $\mathop{d_{E}} w$ satisfies the boundary
    condition, whenever $w\in \Omega_{\Delta_{\mathcal{B}}}(M;E)(\lambda)$. On $\partial_{+}M$,
    the absolute part of the boundary, this immediately follows
    from $\mathop{d_{E}}^2=0$. Let us turn to $\partial_{-}M$, the relative part of the boundary.
    But, we know that the Riesz projections are well defined as bounded operators and they
    commute with the Laplacian on its domain of definition.
    That is, $\Delta_{E,g,b}w$
    lies in $\Omega_{\Delta_{\mathcal{B}}}(M;E)(\lambda)$ as well; in particular,
    it satisfies relative boundary conditions on $\partial_{-}M$, so that  $i^{*}_{-}(\Delta_{E,g,b}w)=0$. Together
    with $i^{*}_{-}\mathop{d^{\sharp}_{E,g,b}}w=0$, this implies $i^{*}_{-}\mathop{d^{\sharp}_{E,g,b}} \mathop{d_{E}} w=0$,
    hence $\mathop{d_{E}} w$ also satisfies relative boundary conditions. Finally, the corresponding statement for
    $\mathop{d^{\sharp}_{E,g,b}}$ follows by the duality between the absolute and relative boundary operators.
    \end{proof}

\subsection{Orthogonality and Hodge decomposition for smooth forms}
\label{section_Orthogonal decomposition}
We are interested in the space of smooth forms being in the complement image of $\mathsf{P}_{\mathcal{B}}(\lambda)$, which is denoted by
    \begin{equation}
    \label{equation_complement_image_smooth_forms}
     {\Omega_{\Delta_{\mathcal{B}}}(M;E)(\lambda)^{\mathsf{c}}}:=\Omega(M;E)\cap \mathsf{Im}(\mathsf{Id}-\mathsf{P}_{\Delta_{\mathcal{B}}}(\lambda)).
    \end{equation}
    Invertibility of the operator given in (\ref{equation_restriction_of_delta_on_complement}) and the existence
    of elliptic estimates imply that
    the restriction of $\left(\Delta_{\mathcal{B}}-\lambda\right)$ to the space
    $\Omega_{\Delta_{\mathcal{B}}}(M;E)(\lambda)^{\mathsf{c}}$ given in (\ref{equation_complement_image_smooth_forms}), satisfying boundary conditions
    provides, with the notation in display (\ref{Definition_subspace_space_of_smooth_forms_with_boundary_conditions}), the isomorphism
    \begin{equation}
    \label{Lemma_invertibility_of_Delta_in_smooth_subspaces}
    \left(\Delta_{\mathcal{B}}-\lambda\right)|_{{\Omega_{\Delta_{\mathcal{B}}}(M;E)(\lambda)^{\mathsf{c}}}|_{\mathcal{B}}}:
    {\Omega_{\Delta_{\mathcal{B}}}(M;E)(\lambda)^{\mathsf{c}}}|_{\mathcal{B}}\rightarrow{\Omega_{\Delta_{\mathcal{B}}}(M;E)(\lambda)^{\mathsf{c}}}.
    \end{equation}
    \begin{lem}
        \label{symmetry_of_beta_wrt_projection}
    For $\lambda\in \mathsf{Spec}(\Delta_{\mathcal{B}})$ and $v,w\in L^{2}(M;E)$, we have the formula
    $\beta_{g,b}(\mathsf{P}_{\Delta_{\mathcal{B}}}(\lambda)v,w)=\beta_{g,b}(v,\mathsf{P}_{\Delta_{\mathcal{B}}}(\lambda)w).$
    \end{lem}
    \begin{proof} Since  $\beta_{g,b}$ continuously extends to a nondegenerate bilinear form on $L^{2}(M;E)$,
    it is enough to prove the statement on smooth forms. For  $v,w\in\Omega(M;E)$ and the definition of the spectral projection in (\ref{Riesz_Projection}), we have
    $$
      \begin{array}{lcr}{ }
    -2\pi \mathsf{i}\beta_{g,b}(\mathsf{P}_{\Delta_{\mathcal{B}}}(\lambda)v,w)&=&\beta_{g,b}\left(
    \int_{\gamma_{\lambda}}(\Delta_{\mathcal{B}}-\mu)^{-1}v \mathop{d\mu},w \right)\\
    &=&\int_{\gamma_{\lambda}}\beta_{g,b}\left((\Delta_{\mathcal{B}}-\mu)^{-1}v,w\right)\mathop{d\mu},
    \end{array}
    $$
    where the last equality above holds,
    since $\int_{\gamma_{\lambda}}$ converges uniformly in the
    $L^{2}$-norm.
    Since $\gamma_{\lambda}\cap\mathsf{Spec}(\Delta_{\mathcal{B}})=\emptyset$, we have
    $(\Delta_{\mathcal{B}}-\mu)^{-1}w\in\mathcal{D}(\Delta_{\mathcal{B}})$ so that
    $w=(\Delta_{\mathcal{B}}-\mu)(\Delta_{\mathcal{B}}-\mu)^{-1}w$ for each $\mu\in\gamma_{\lambda}$.
    Now, from the isomorphism in (\ref{Lemma_invertibility_of_Delta_in_smooth_subspaces}), both
    $(\Delta_{\mathcal{B}}-\mu)^{-1}v$ and $(\Delta_{\mathcal{B}}-\mu)^{-1}w$ belong in fact to
    ${\Omega_{\Delta_{\mathcal{B}}}(M;E)(\lambda)^{\mathsf{c}}}|_{\mathcal{B}}$,
    so we can apply Lemma \ref{Lemma_image_by_d_of_omega_2_and_symmetry_beta_boundary_conditions} and obtain
    $$
    \begin{array}{lll}
    { }\beta_{g,b}\left((\Delta_{\mathcal{B}}-\mu)^{-1}v,w\right)&{ }=&
    { }\beta_{g,b}((\Delta_{\mathcal{B}}-\mu)^{-1}v,(\Delta_{E,g,b}-\mu)(\Delta_{\mathcal{B}}-\mu)^{-1}w)\\
    &{ }=&{ }\beta_{g,b}\left((\Delta_{E,g,b}-\mu)(\Delta_{\mathcal{B}}-\mu)^{-1}v,(\Delta_{\mathcal{B}}-\mu)^{-1}w\right)\\
    &{ }=&{ }\beta_{g,b}\left(v,(\Delta_{\mathcal{B}}-\mu)^{-1}w\right);
    \end{array}
    $$
    that is,
    $
    \beta_{g,b}(\mathsf{P}_{\Delta_{\mathcal{B}}}(\lambda)v,w)=-(-2\pi \mathsf{i})^{-1}\int_{\gamma_{\lambda}}\beta_{g,b}(v,(\Delta_{\mathcal{B}}-\mu)^{-1}w)\mathop{d\mu}
    $
    and hence the equality
    $\beta_{g,b}(\mathsf{P}_{\Delta_{\mathcal{B}}}(\lambda)v,w)=\beta_{g,b}(v,\mathsf{P}_{\Delta_{\mathcal{B}}}(\lambda)w)$ holds.
    \end{proof}

    \begin{prop}
    \label{Proposition_orthogonal_decomposition_smooth_forms}
    There is a $\beta_{g,b}$-orthogonal direct sum decomposition:
    \begin{equation}
    \label{equation_decomposition_forms}
    \Omega(M;E)\cong\Omega_{\Delta_{\mathcal{B}}}(M;E)(\lambda)\oplus{\Omega_{\Delta_{\mathcal{B}}}(M;E)(\lambda)}^{\mathsf{c}}.
    \end{equation}
    If $\lambda,\mu\in \mathsf{Spec}(\Delta_{\mathcal{B}})$ with $\lambda\not=\mu$, then
    $
    \Omega_{\Delta_{\mathcal{B}}}(M;E)(\mu)\perp_{\beta}\Omega_{\Delta_{\mathcal{B}}}(M;E)(\lambda).
    $ In particular, $\beta_{g,b}$ restricts to each of these subspaces as a non degenerate symmetric bilinear form.
    Furthermore, with the notation in Section \ref{section_Absolute_and_relative_boundary_conditions_on_bordisms},
    there is a $\beta_{g,b}$-orthogonal direct sum decomposition
    \begin{equation}
    \label{equation_decomposition_forms_with_relative_boundary_conditions_on_minus}
    \Omega(M;E)|_{\mathcal{B}_{-}^{0}}\cong\Omega_{\Delta_{\mathcal{B}}}(M;E)(\lambda)
    \oplus{\Omega_{\Delta_{\mathcal{B}}}(M;E)(\lambda)}^{\mathsf{c}}|_{\mathcal{B}_{-}^{0}},
    \end{equation}
    which is invariant under $\mathop{d_{E}}$.
    \end{prop}
    \begin{proof}
    Remark that
    $
    \Omega_{\Delta_{\mathcal{B}}}(M;E)(\lambda)=\mathsf{P}_{\Delta_{\mathcal{B}}}(\lambda)(\Omega(M;E)).
    $
    Therefore the decomposition in (\ref{equation_decomposition_forms}) follows from the direct sum decomposition of $L^{2}(M;E)$ stated in Lemma \ref{Kato}.
    We show that $\Omega_{\Delta_{\mathcal{B}}}(M;E)(\lambda)$ is $\beta_{g,b}$-orthogonal to ${\Omega_{\Delta_{\mathcal{B}}}(M;E)(\lambda)^{\mathsf{c}}}$, by
    taking $v\in\Omega_{\Delta_{\mathcal{B}}}(M;E)(\lambda)$ and
   $w\in{\Omega_{\Delta_{\mathcal{B}}}(M;E)(\lambda)^{\mathsf{c}}}$ and noticing that
    $$
    \beta_{g,b}(v,w)=\beta_{g,b}(\mathsf{P}_{\Delta_{\mathcal{B}}}(\lambda) v,w)
                    =\beta_{g,b}(v,\mathsf{P}_{\Delta_{\mathcal{B}}}(\lambda)w)
                    =0,
    $$
    where the second equality above follows from Lemma \ref{symmetry_of_beta_wrt_projection} and the
    last one is true because $w$  is in the image of the complementary projection of $\mathsf{P}_{\Delta_{\mathcal{B}}}(\lambda)$.
    Since $\Omega_{\Delta_{\mathcal{B}}}(M;E)(\lambda)$ is contained in the space $\Omega(M;E)|_{\mathcal{B}_{-}^{0}}$, the decomposition
    in (\ref{equation_decomposition_forms}) implies directness and $\beta_{g,b}$-orthogonality for the one in
    (\ref{equation_decomposition_forms_with_relative_boundary_conditions_on_minus}).
    By Lemma
    \ref{Proposition_finite_dimensional_generalized_eigen_spaces},
    $\Omega_{\Delta_{\mathcal{B}}}(M;E)(\lambda)$ is invariant under both $\mathop{d_{E}}$ and $\mathop{d^{\sharp}_{E,g,b}}$.
    But, the space
    $\mathop{d_{E}}({\Omega_{\Delta_{\mathcal{B}}}(M;E)(\lambda)}^{\mathsf{c}}|_{\mathcal{B}_{-}^{0}})$ is contained
    in ${\Omega_{\Delta_{\mathcal{B}}}(M;E)(\lambda)}^{\mathsf{c}}|_{\mathcal{B}_{-}^{0}}$ as well,
    as it can be checked by using the Green's formulas from Lemma
    \ref{Lemma_Stokes_theorem_transpose_Greens_Formulas_non_self_adjoint_Laplacian}, that
    $\mathop{d^{\sharp}_{E,g,b}}$ leaves invariant
    $\Omega_{\Delta_{\mathcal{B}}}(M;E)(\lambda)$ and $\beta_{g,b}$-orthogonality
    of (\ref{equation_decomposition_forms}).
    \end{proof}

    \begin{cor}
    \label{Corollary_Beta_orthogonality_d_d_sharp}
    For $\lambda\in\mathsf{Spec}(\Delta_{\mathcal{B}})$ and with the notation in (\ref{Definition_subspace_space_of_smooth_forms_with_boundary_conditions}),
    consider the space ${\Omega_{\Delta_{\mathcal{B}}}(M;E)(\lambda)^{\mathsf{c}}}|_{\mathcal{B}^{0}}$. Then,
    the spaces
    $\mathop{d_{E}}({\Omega_{\Delta_{\mathcal{B}}}(M;E)(\lambda)^{\mathsf{c}}}|_{\mathcal{B}^{0}})$
    and
    $\mathop{d^{\sharp}_{E,g,b}}({\Omega_{\Delta_{\mathcal{B}}}(M;E)(\lambda)^{\mathsf{c}}}|_{\mathcal{B}^{0}})$ are
    $\beta_{g,b}$-orthogonal to  $\Omega_{\Delta_{\mathcal{B}}}(M;E)(\lambda).$
    \end{cor}
    \begin{proof}
    If $u\in\Omega_{\Delta_{\mathcal{B}}}(M;E)(\lambda)$ and $v\in {\Omega_{\Delta_{\mathcal{B}}}(M;E)(\lambda)^{\mathsf{c}}}|_{\mathcal{B}^{0}}$,
    then, by using Lemma \ref{Lemma_image_by_d_of_omega_2_and_symmetry_beta_boundary_conditions}, invariance of
    $\Omega_{\Delta_{\mathcal{B}}}(M;E)(\lambda)$ under $\mathop{d^{\sharp}_{E,g,b}}$ (see also
    Lemma \ref{Proposition_finite_dimensional_generalized_eigen_spaces} and
    Proposition \ref{Proposition_orthogonal_decomposition_smooth_forms} above), we have
    $\beta_{g,b}(u,\mathop{d_{E}}v)=\beta_{g,b}(\mathop{d^{\sharp}_{E,g,b}}u,v)=0.$
    The proof for $\mathop{d^{\sharp}_{E,g,b}}$ is analog.
     \end{proof}

\begin{cor}(Hodge decomposition)
    \label{Corollary_Hodge_decomposition}
    We have the $\beta_{g,b}$-orthogonal decomposition
    $
    \label{Theorem_Hodge_decomposition_formula_1}
    \Omega(M;E)\cong\Omega_{\Delta_{\mathcal{B}}}(M;E)(0)
    \oplus\Delta_{E,g,b}(\left.{\Omega_{\Delta_{\mathcal{B}}}(M;E)(0)^{\mathsf{c}}}\right.|_{\mathcal{B}}).
    $
    \end{cor}
    \begin{proof}
     This follows from Proposition
     \ref{Proposition_orthogonal_decomposition_smooth_forms} and
     the isomorphism in (\ref{Lemma_invertibility_of_Delta_in_smooth_subspaces}).
    \end{proof}
Compare the folloring result with
\cite[Proposition 2.1]{Burghelea_Friedlander_Kappeler}.
\begin{prop}
\label{Proposition_from_Hodge_decomposition_2}
The following are $\beta_{g,b}$-orthogonal direct sum decompositions.
\begin{eqnarray}
		\nonumber \Omega(M;E)&\cong&\Omega_{\Delta_{\mathcal{B}}}(M;E)(0)  \oplus  \mathop{d_{E}}(\mathop{d^{\sharp}_{E,g,b}}(\left.{\Omega_{\Delta_{\mathcal{B}}}(M;E)(0)^{\mathsf{c}}}\right.|_{\mathcal{B}}))\\
		 &&\label{Proposition_from_Hodge_decomposition_2_Decomposition_of_forms_from_Theorem_Hodge_decomposition_1}\phantom{\Omega_{\Delta_{\mathcal{B}}}(M;E)(0)}\oplus
                \mathop{d^{\sharp}_{E,g,b}}(\mathop{d_{E}}(\left.{\Omega_{\Delta_{\mathcal{B}}}(M;E)(0)^{\mathsf{c}}}\right.|_{\mathcal{B}}),\\
		\nonumber \Omega(M;E)|_{\mathcal{B}_{-}^{0}}&\cong&\Omega_{\Delta_{\mathcal{B}}}(M;E)(0)\oplus \mathop{d_{E}}(\left.{\Omega_{\Delta_{\mathcal{B}}}(M;E)(0)^{\mathsf{c}}}\right.|_{\mathcal{B}_{-}^{0}})\\
                    &&\label{Proposition_from_Hodge_decomposition_2_Decomposition_of_forms_in_domain_of_d_E}\phantom{\Omega_{\Delta_{\mathcal{B}}}(M;E)(0)}\oplus
                \mathop{d^{\sharp}_{E,g,b}}(\left.{\Omega_{\Delta_{\mathcal{B}}}(M;E)(0)^{\mathsf{c}}}\right.|_{\mathcal{B}}),\\
		\nonumber\Omega(M;E)|_{\mathcal{B}^{0}}&\cong&\Omega_{\Delta_{\mathcal{B}}}(M;E)(0)\oplus \mathop{d_{E}}(\left.{\Omega_{\Delta_{\mathcal{B}}}(M;E)(0)^{\mathsf{c}}}\right.|_{\mathcal{B}})\\
                    &&\label{Proposition_from_Hodge_decomposition_2_Decomposition_of_forms_satisfying_1}\phantom{\Omega_{\Delta_{\mathcal{B}}}(M;E)(0)}\oplus
                \mathop{d^{\sharp}_{E,g,b}}(\left.{\Omega_{\Delta_{\mathcal{B}}}(M;E)(0)^{\mathsf{c}}}\right.|_{\mathcal{B}}).
\end{eqnarray}
Moreover, the restriction of $\beta_{g,b}$ to each of the spaces appearing above is nondegenerate.

\end{prop}

\begin{proof}
    We prove (\ref{Proposition_from_Hodge_decomposition_2_Decomposition_of_forms_from_Theorem_Hodge_decomposition_1}).
    From Corollary \ref{Corollary_Hodge_decomposition},  every $u\in\Omega(M;E)$ can be written as
    $u=u_0+\mathop{d_{E}}(\mathop{d^{\sharp}_{E,g,b}}u)+\mathop{d^{\sharp}_{E,g,b}}(\mathop{d_{E}}u),$
    with $u_0\in\Omega_{\Delta_{\mathcal{B}}}(M;E)(0)$ and
    $u\in\left.{\Omega_{\Delta_{\mathcal{B}}}(M;E)(0)^{\mathsf{c}}}\right.|_{\mathcal{B}}$.
    That
    $$
    \mathop{d_{E}}(\mathop{d^{\sharp}_{E,g,b}}(\left.{\Omega_{\Delta_{\mathcal{B}}}(M;E)(0)^{\mathsf{c}}}\right.|_{\mathcal{B}}))
    \perp_{\beta_{g,b}}
    \mathop{d^{\sharp}_{E,g,b}}(\mathop{d_{E}}(\left.{\Omega_{\Delta_{\mathcal{B}}}(M;E)(0)^{\mathsf{c}}}\right.|_{\mathcal{B}})),
    $$
    follows from Lemma
    \ref{Lemma_image_by_d_of_omega_2_and_symmetry_beta_boundary_conditions}
    and $\mathop{d_{E}}^2=0$.
    To see that (\ref{Proposition_from_Hodge_decomposition_2_Decomposition_of_forms_from_Theorem_Hodge_decomposition_1})
    is a direct sum,
    we check that the intersection of the last two spaces on the right of (\ref{Proposition_from_Hodge_decomposition_2_Decomposition_of_forms_from_Theorem_Hodge_decomposition_1})
    is trivial. So, take $u\in\Omega_{\Delta_{\mathcal{B}}}(M;E)(0)^{\mathsf{c}}$,
    and suppose there are $v, w\in\left.{\Omega_{\Delta_{\mathcal{B}}}(M;E)(0)^{\mathsf{c}}}\right.|_{\mathcal{B}}$
    with $u=\mathop{d_{E}}(\mathop{d^{\sharp}_{E,g,b}}v)=\mathop{d^{\sharp}_{E,g,b}}(\mathop{d_{E}}w)$. Remark obviously that
    $
    \Delta_{E,g,b}u=0$
    but also that $u\in\Omega_{\Delta_{\mathcal{B}}}(M;E)(0)$, since
    \begin{itemize}
     \item[(a)] $
    i^{*}_{-}u= %i^{*}_{-}\mathop{d_{E}}\mathop{d^{\sharp}_{E,g,b}}v=
    \mathop{d_{E}}(i^{*}_{-}\mathop{d^{\sharp}_{E,g,b}}v)=0,
    $ as $v$ satisfies boundary conditions,
     \item[(b)]  $
    i^{*}_{-}\mathop{d^{\sharp}_{E,g,b}}u= i^{*}_{-}\mathop{d^{\sharp}_{E,g,b}}\mathop{d^{\sharp}_{E,g,b}}\mathop{d_{E}}v=0,
    $
    \item [(c)]
	  $
           i^{*}_{+}\star_{b}u= %i^{*}_{+}\star_{b}\mathop{d^{\sharp}_{E,g,b}}\mathop{d_{E}}w
			      %=\pm i^{*}_{+}\star_{b}\star_{b}\mathop{d_{E}}\star_{b}\mathop{d_{E}}w
			      %=\pm i^{*}_{+}\mathop{d_{E}}\star_{b}\mathop{d_{E}}w
			      %&=&\pm \mathop{d_{E}}i^{*}_{+}\star_{b}\mathop{d_{E}}w
			      %&=&\pm \mathop{d_{E}}i^{*}_{+}\star_{b}\mathop{d_{E}}\star_{b}\star_{b}w
			      \pm \mathop{d_{E}}(i^{*}_{+}\mathop{d^{\sharp}_{E,g,b}}\star_{b}w)=0;
	  $ as $w$ satisfies boundary conditions,
	\item [(d)]
	  $
		i^{*}_{+}\mathop{d^{\sharp}_{E,g,b}}\star_{b}u=%\pm i^{*}_{+}\star_{b}\mathop{d_{E}}\star_{b}\star_{b}u=
                 %\pm i^{*}_{+}\star_{b}\mathop{d_{E}}u=
		  \pm i^{*}_{+}\star_{b}\mathop{d_{E}}(\mathop{d_{E}}\mathop{d^{\sharp}_{E,g,b}}v)=0;
	  $
    \end{itemize}
    therefore, from Proposition
    \ref{Proposition_orthogonal_decomposition_smooth_forms}, $u$ must vanish, so that the sum in
    (\ref{Proposition_from_Hodge_decomposition_2_Decomposition_of_forms_from_Theorem_Hodge_decomposition_1}) is direct.
    This decomposition is clearly $\beta_{g,b}$-orthogonal.
    The decompositions in (\ref{Proposition_from_Hodge_decomposition_2_Decomposition_of_forms_in_domain_of_d_E})
    and (\ref{Proposition_from_Hodge_decomposition_2_Decomposition_of_forms_satisfying_1}) follow from
    that in (\ref{Proposition_from_Hodge_decomposition_2_Decomposition_of_forms_from_Theorem_Hodge_decomposition_1}),
    Lemma \ref{Lemma_image_by_d_of_omega_2_and_symmetry_beta_boundary_conditions}, the isomorphism in
    (\ref{Lemma_invertibility_of_Delta_in_smooth_subspaces}) and the definition of boundary conditions as we have proceeded
    to prove the statement
    (\ref{Proposition_from_Hodge_decomposition_2_Decomposition_of_forms_from_Theorem_Hodge_decomposition_1}); we omit the details.
    Now, since
    $
    \mathop{d_{E}}(\left.{\Omega_{\Delta_{\mathcal{B}}}(M;E)(0)^{\mathsf{c}}}\right.|_{\mathcal{B}})\subset
                \mathop{d_{E}}(\left.{\Omega_{\Delta_{\mathcal{B}}}(M;E)(0)^{\mathsf{c}}}\right.|_{\mathcal{B}^{0}_{-}}),
    $
    directness of decomposition
    (\ref{Proposition_from_Hodge_decomposition_2_Decomposition_of_forms_satisfying_1}) follows from that of
    (\ref{Proposition_from_Hodge_decomposition_2_Decomposition_of_forms_in_domain_of_d_E}). To check directness
    in (\ref{Proposition_from_Hodge_decomposition_2_Decomposition_of_forms_in_domain_of_d_E}), firstly observe that by Proposition
    \ref{Proposition_orthogonal_decomposition_smooth_forms} we have
    $\mathop{d_{E}}({\Omega_{\Delta_{\mathcal{B}}}(M;E)(0)^{\mathsf{c}}}|_{\mathcal{B}_{-}^{0}})
    \subset{\Omega_{\Delta_{\mathcal{B}}}(M;E)(0)^{\mathsf{c}}}|_{\mathcal{B}_{-}^{0}}$ and therefore the intersection of the space
    $\Omega_{\Delta_{\mathcal{B}}}(M;E)(0)$ with $
    \left. \mathop{d_{E}}\right.(\left.{\Omega_{\Delta_{\mathcal{B}}}(M;E)(0)^{\mathsf{c}}}\right.|_{\mathcal{B}_{-}^{0}})$ is trivial.
    Secondly, from the inclusion
  $\left.{\Omega_{\Delta_{\mathcal{B}}}(M;E)(0)^{\mathsf{c}}}\right.|_{\mathcal{B}}\subset
  \left.{\Omega_{\Delta_{\mathcal{B}}}(M;E)(0)^{\mathsf{c}}}\right.|_{\mathcal{B}^{0}}$,
  Corollary \ref{Corollary_Beta_orthogonality_d_d_sharp} and
  Proposition \ref{Proposition_orthogonal_decomposition_smooth_forms},  the intersection of
  $\Omega_{\Delta_{\mathcal{B}}}(M;E)(0)$ with the space $
  \mathop{d^{\sharp}_{E,g,b}}(\left.{\Omega_{\Delta_{\mathcal{B}}}(M;E)(0)^{\mathsf{c}}}\right.|_{\mathcal{B}})$ is also trivial.
  Thirdly, the intersection between $\left. \mathop{d_{E}}\right.(\left.{\Omega_{\Delta_{\mathcal{B}}}(M;E)(0)^{\mathsf{c}}}\right.|_{\mathcal{B}_{-}^{0}})$
  and  $\mathop{d^{\sharp}_{E,g,b}}(\left.{\Omega_{\Delta_{\mathcal{B}}}(M;E)(0)^{\mathsf{c}}}\right.|_{\mathcal{B}})$
  is trivial as well; indeed, if
   %\textit{non zero}
  $u\in\Omega_{\Delta_{\mathcal{B}}}(M;E)(0)^{\mathsf{c}}$ with
  $u=\mathop{d_{E}}v$ for certain $v\in\left.{\Omega_{\Delta_{\mathcal{B}}}(M;E)(0)^{\mathsf{c}}}\right.|_{\mathcal{B}_{-}^{0}}$ and
  $u=\mathop{d^{\sharp}_{E,g,b}}w$ for
  $w\in \mathop{d^{\sharp}_{E,g,b}}(\left.{\Omega_{\Delta_{\mathcal{B}}}(M;E)(0)^{\mathsf{c}}}\right.|_{\mathcal{B}})$, then, it is
  follows that
  $u\in\Omega_{\Delta_{\mathcal{B}}}(M;E)(0)$, and therefore $u=0$.
  Finally, the bilinear form $\beta_{g,b}$ is nondegenerate on each of the spaces appearing in the direct sum decompositions (i), (ii) and (iii).
  Indeed, on the one hand, $\beta_{g,b}$ is nondegenerate on each of the spaces appearing on the left hand side of the equalities (i), (ii) and (iii),
  exactly for the same reason as $\beta_{g,b}$ is nondegenerate on
  $\Omega_{0}(M;E)$, the
  space of smooth forms compactly supported in the interior of $M$; this follows immediately from the requirement for $b$ to be fiberwise nondegenerate on $E$.
  On the other hand, from Lemma
 \ref{Lemma_image_by_d_of_omega_2_and_symmetry_beta_boundary_conditions},
 the direct sum decompositions in (\ref{Proposition_from_Hodge_decomposition_2_Decomposition_of_forms_from_Theorem_Hodge_decomposition_1}),
 (\ref{Proposition_from_Hodge_decomposition_2_Decomposition_of_forms_in_domain_of_d_E})
 and (\ref{Proposition_from_Hodge_decomposition_2_Decomposition_of_forms_satisfying_1}) are $\beta_{g,b}$-orthogonal. Thus,
 $\beta_{g,b}$ restricts to each space appearing on the right hand side of
 (\ref{Proposition_from_Hodge_decomposition_2_Decomposition_of_forms_from_Theorem_Hodge_decomposition_1}),
 (\ref{Proposition_from_Hodge_decomposition_2_Decomposition_of_forms_in_domain_of_d_E})
 and (\ref{Proposition_from_Hodge_decomposition_2_Decomposition_of_forms_satisfying_1})
 as a nondegenerate bilinear form as well.
\end{proof}

\subsection{Cohomology}
\label{subsection_cohomology}
 Recall the notation suggested in Lemma \ref{Lemma_image_by_d_of_omega_2_and_symmetry_beta_boundary_conditions}.
 The space
   $\Omega(M;E)|_{\mathcal{B}_{-}^{0}}$ endowed with the differential $\mathop{d_{E}}$ is a cochain complex, which
  computes De-Rham cohomology of $M$ relative to
  $\partial_{-}M$
  with coefficients on $E$,
  see for instance \cite{Bott-Tu}.
  For $\lambda\in\mathsf{Spec}(\Delta_{\mathcal{B}})$, consider
 $\Omega_{\Delta_{\mathcal{B}}}(M;E)(\lambda)$ as a cochain subcomplex of $\Omega(M;E)|_{\mathcal{B}_{-}^{0}}$. %Then consider the inclusion
 %$$\Omega_{\Delta_{\mathcal{B}}}(M;E)(\lambda)\hookrightarrow\Omega(M;E)|_{\mathcal{B}_{-}^{0}}.$$
 From Lemma \ref{Kato}, Lemma \ref{Proposition_finite_dimensional_generalized_eigen_spaces}
 and the isomorphism in (\ref{Lemma_invertibility_of_Delta_in_smooth_subspaces}), every generalized eigenspace
 corresponding to a \textit{non-zero} eigenvalue is acyclic, i.e.,
 $
 H(\Omega_{\Delta_{\mathcal{B}}}(M;E)(\lambda))=0
 $
 whenever $\lambda\not=0$.
 For $\lambda=0$, we have the following.
 \begin{prop}
    \label{Proposition_Absolute_relative_cohomology}
The inclusion
 $\Omega_{\Delta_{\mathcal{B}}}(M;E)(0)\hookrightarrow\Omega(M;E)|_{\mathcal{B}_{-}^{0}}$
 induces an isomorphism in cohomology:
    $
    H^{*}(\Omega_{\Delta_{\mathcal{B}}}(M;E)(0))\cong
    H^{*}(M,\partial_{-}M,E).
    $
\end{prop}
\begin{proof}
Since $\Omega_{\Delta_{\mathcal{B}}}(M;E)(0)\subset\Omega(M;E)|_{\mathcal{B}_{-}^{0}}$, the space
$\Omega(M;E)|_{\mathcal{B}_{-}^{0}}$ admits a decomposition compatible with the one in
Corollary \ref{Corollary_Hodge_decomposition} and therefore it decomposes as
$$
\Omega(M;E)|_{\mathcal{B}_{-}^{0}}\cong\Omega_{\Delta_{\mathcal{B}}}(M;E)(0)
\oplus \left.\Delta_{E,g,b}(\left.{\Omega_{\Delta_{\mathcal{B}}}(M;E)(0)^{\mathsf{c}}}\right.|_{\mathcal{B}})\right.|_{\mathcal{B}_{-}^{0}},
$$
where
$
\left.\Delta_{E,g,b}(\left.{\Omega_{\Delta_{\mathcal{B}}}(M;E)(0)^{\mathsf{c}}}\right.|_{\mathcal{B}})\right.|_{\mathcal{B}_{-}^{0}}
%:=\Delta_{E,g,b}(\left.{\Omega_{\Delta_{\mathcal{B}}}(M;E)(0)^{\mathsf{c}}}\right.|_{\mathcal{B}})\cap\Omega(M;E)|_{\mathcal{B}_{-}^{0}}
$
is also a cochain subcomplex,
because of Proposition \ref{Proposition_orthogonal_decomposition_smooth_forms}
and that $\Omega(M;E)|_{\mathcal{B}_{-}^{0}}$ is invariant under the action of $\mathop{d_{E}}$.
Thus the assertion is true, if the corresponding cohomology groups vanish; that is, if
 %$H^{*}(\left.\Delta_{E,g,b}(\left.{\Omega_{\Delta_{\mathcal{B}}}(M;E)(0)^{\mathsf{c}}}\right.|_{\mathcal{B}})\right.|_{\mathcal{B}_{-}^{0}})$ vanishes.
every closed form $w$ in $\left.\Delta_{E,g,b}(\left.{\Omega_{\Delta_{\mathcal{B}}}(M;E)(0)^{\mathsf{c}}}\right.|_{\mathcal{B}})\right.|_{\mathcal{B}_{-}^{0}}$
is also exact.
By
Proposition \ref{Proposition_from_Hodge_decomposition_2}.(\ref{Proposition_from_Hodge_decomposition_2_Decomposition_of_forms_in_domain_of_d_E}), there exist
 $w_{1}\in\left.{\Omega_{\Delta_{\mathcal{B}}}(M;E)(0)^{\mathsf{c}}}\right.|_{\mathcal{B}_{-}^{0}}$ and
$w_2\in\left.{\Omega_{\Delta_{\mathcal{B}}}(M;E)(0)^{\mathsf{c}}}\right.|_{\mathcal{B}}$ such that
$w=\mathop{d_{E}}w_1+\mathop{d^{\sharp}_{E,g,b}}w_2$. First,
we claim that
$\beta_{g,b}(\mathop{d^{\sharp}_{E,g,b}}w_{2},v_{1})=0,$
 for all $v_{1}\in\left.{\Omega_{\Delta_{\mathcal{B}}}(M;E)(0)^{\mathsf{c}}}\right.|_{\mathcal{B}^{0}}$, see
(\ref{Definition_subspace_space_of_smooth_forms_with_boundary_conditions}); indeed,
from Proposition
\ref{Proposition_from_Hodge_decomposition_2}.(\ref{Proposition_from_Hodge_decomposition_2_Decomposition_of_forms_from_Theorem_Hodge_decomposition_1}),
there exist
$v_{2},u_{2}\in\left.{\Omega_{\Delta_{\mathcal{B}}}(M;E)(0)^{\mathsf{c}}}\right.|_{\mathcal{B}}$, such that
$v_1=\mathop{d_{E}}v_{2}+\mathop{d^{\sharp}_{E,g,b}}u_{2}$ and hence
$%{ }
\beta_{g,b}(\mathop{d^{\sharp}_{E,g,b}}w_2,\mathop{d_{E}}v_{2}+\mathop{d^{\sharp}_{E,g,b}}u_{2})=
%\beta_{g,b}(\mathop{d^{\sharp}_{E,g,b}}w_2,\mathop{d_{E}}v_{2})+\beta_{g,b}(\mathop{d^{\sharp}_{E,g,b}}w_2,\mathop{d^{\sharp}_{E,g,b}}u_{2})=
0,
$
where we have used that $\mathop{d^{\sharp}_{E,g,b}}w_2$, $\mathop{d_{E}}v_{2} $ and
$\mathop{d^{\sharp}_{E,g,b}}u_{2}\in\left.{\Omega_{\Delta_{\mathcal{B}}}(M;E)(0)^{\mathsf{c}}}\right.|_{\mathcal{B}^{0}}$,
Lemma
\ref{Lemma_image_by_d_of_omega_2_and_symmetry_beta_boundary_conditions}, $(\mathop{d^{\sharp}_{E,g,b}})^{2}=0$ and that
$\beta_{g,b}(\mathop{d_{E}}\mathop{d^{\sharp}_{E,g,b}}w_2,u_{2})$ vanishes, because $w$ being close implies $\mathop{d_{E}}\mathop{d^{\sharp}_{E,g,b}}w_2=0$.
Finally, since $\mathop{d^{\sharp}_{E,g,b}}w_2$ belongs to $\left.{\Omega_{\Delta_{\mathcal{B}}}(M;E)(0)^{\mathsf{c}}}\right.|_{\mathcal{B}^{0}}$ as well,
and that $\beta_{g,b}$ restricted to this sub-space is also nondegenerate,
see Proposition \ref{Proposition_from_Hodge_decomposition_2},
from the claim above, we have $\mathop{d^{\sharp}_{E,g,b}}w_2=0$.
That is, $w$ is exact in
$\left.\Delta_{E,g,b}(\left.{\Omega_{\Delta_{\mathcal{B}}}(M;E)(0)^{\mathsf{c}}}\right.|_{\mathcal{B}})\right.|_{\mathcal{B}_{-}^{0}}$.

\end{proof}

\section{Heat trace asymptotic expansion and anomaly formulas}
\label{section_Heat_asymptotics and anomaly formulas}
\subsection{Heat trace asymptotics for an elliptic boundary value problem}
\label{subsection_Heat trace asymptotics for the heat kernel of a Laplace type operator under local boundary conditions}
 Let $(\mathsf{D},\mathcal{B})$ be a boundary value problem,
 where $\mathsf{D}$ is an operator of Laplace type and
 $\mathcal{B}$ is a boundary operator specifying absolute/relative boundary conditions, (or more generally \emph{mixed boundary conditions}, see \cite{Gilkey2}) and denote by $\mathsf{D}_{\mathcal{B}}$ its $L^{2}$-realization, see Section \ref{section_The spectrum of the bilinear Laplacian on manifolds with boundary}. Then, by \cite[Theorem 1.4.5]{Gilkey2}, for $t>0$ the heat kernel
 $\exp(-t\mathsf{D}_{\mathcal{B}})$ is a smoothing operator, of trace class in $L^{2}$-norm and for
 $t\rightarrow 0$, there is a complete asymptotic expansion:
 $$
    \left.\Tr\right._{L^{2}}(\psi \exp(-t\mathsf{D}_{\mathcal{B}}))\sim\sum^{\infty}_{n=0}a_{n}(\psi,\mathsf{D},\mathcal{B}) t^{(n-m)/2},
 $$
 where $\psi$ is a bundle endomorphism.
 The coefficients  $a_{n}(\psi,\mathsf{D},\mathcal{B})$,
 the \textit{heat trace asymptotic coefficients associated to $\psi$ and the boundary value problem} $(\mathsf{D},\mathcal{B})$, are given by the formula
  \begin{equation}
  \label{equation_coefficients_asymptotic_expansion}
  \scriptstyle a_{n}(\psi,\mathsf{D},\mathcal{B})=\scriptstyle
  \int_{M}\left.\Tr\right.(\psi\cdot\mathfrak{e}_{n}(\mathsf{D}))\mathsf{vol}_{g}(M)+
  \sum_{k=0}^{n-1}\int_{\partial M}\Tr\left({\nabla_{\varsigma_{\mathsf{in}}}}^{k}
  \psi\cdot\mathfrak{e}_{n,k}(\mathsf{D},\mathcal{B})\right)\mathsf{vol}_{g}(\partial M),
  \end{equation}
  where ${\nabla_{\varsigma_{\mathsf{in}}}}^{k}$ denotes the $k$-covariant
  derivative along the inwards pointing geodesic unit vector field normal to $\partial M$,
  computed with respect to the Levi--Civit\`a connection on $\Lambda^{*}(T^{*}M)$ and an auxiliary connection on the bundle.
  The quantities $\mathfrak{e}_{n}(x,\mathsf{D})$ and
  $\mathfrak{e}_{n,k}(y,\mathsf{D},\mathcal{B})$ in (\ref{equation_coefficients_asymptotic_expansion}) are
  invariant endomorphism-valued forms locally computable
  as polynomials
  in the jets of the symbol of $\mathsf{D}$ and $\mathcal{B}$, see \cite{Greiner}, \cite{Seeley}, \cite{Seeleya} and \cite{Seeleyb}.
  By using Weyl's theory of invariants, these endomorphism invariants can be
  expressible as universal polynomials in locally computable tensorial objects,
  see %\footnote{
  %The following summarizes Lemmas 3.1.10 and 3.1.11 in section 3.1.8 in \cite{Gilkey2}; see also
  %Sections 3.6,  1.7, 1.9 and 4.8 in \cite{Gilkey1}.}
  \cite[Sections 1.7 and 1.8]{Gilkey2} (see also \cite[Sections 1.7, 1.9 and 4.8]{Gilkey1}) and  \cite[Section 3.1.8]{Gilkey2}.

  We are interested in the \emph{coefficient of the constant term} in the heat asymptotic expansion in (\ref{equation_coefficients_asymptotic_expansion}) corresponding to $n=\textsf{dim}(M)=m$, which in accord with the notation in \cite{Berline-and-Co}, we denote
  by
  \begin{equation}
  \label{equation_notation_for_the_constant_term_in_the asymptotic_expansion}
  \LIM_{t\rightarrow 0}(\left.\Tr\right._{L^{2}}(\psi \exp(-t\mathsf{D}_{\mathcal{B}}))):=a_{m}(\psi,\mathsf{D},\mathcal{B}).
  \end{equation}

\subsection{Heat trace asymptotics for the Hermitian Laplacian}
\label{subSection_Constant term in the heat trace asymptotics of Hermitian boundary value problems}
Br\"uning and Ma studied in \cite{Bruening-Ma} the Hermitian Laplacian
on a manifold with boundary under absolute boundary conditions and obtained anomaly formulas for
the associated Ray--Singer analytic metric.
%Their formulas express the variation of the Ray--Singer metric
%with respect to smooth variations of the metric on $M$ and of the Hermitian structure on $E$.
They do so by computing the coefficient of the constant term in certain heat trace asymptotic expansion associated to the Hermitian boundary value problem.% under absolute
%boundary conditions.

  Proposition
  \ref{Proposition_infinitesimal_constant_term_in_asymptotic_expansion_Hermitian_absolute_boundary_conditions} below is basically due to the work by
  Br\"uning and Ma in \cite{Bruening-Ma}. In order to read its statement,
  we need certain characteristic forms on $M$ and $\partial M$. The forms defined on $M$, already appearing in the 
  anomaly formulas for the torsion in the situation without boundary,
  are the Euler form $\mathbf{e}(M,g)\in\Omega^{m}(M;\Theta_{M})$, associated to the metric $g$, 
  and secondary forms of Chern--Simons type
  $\mathbf{\widetilde{e}}(M,g,g^{\prime})\in\Omega^{m-1}(M;\Theta_{M})$ associated to two (smoothly connected) Riemannian metrics
  $g$ and $g^{\prime}$. The forms defined on $\partial M$, already defined by
  Br\"uning and Ma, are on the one hand $\mathbf{e_{b}}(\partial M,g)$ and $B(\partial M,g)\in\Omega^{m-1}(\partial M;\Theta_{M})$, see \cite[expression (1.17), page 775]{Bruening-Ma} and on the other
  certain Chern--Simons forms  $\mathbf{\widetilde{e}_{b}}(\partial M,g,g^{\prime})\in\Omega^{m-2}(\partial M;\Theta_{M})$,
  see \cite[expression (1.45), page 780]{Bruening-Ma}.
  For the sake of completeness, we recall in the Appendix, how these characteristic forms
  were constructed in \cite{Bruening-Ma}.

\begin{prop}\textbf{(Br\"uning--Ma)}
 \label{Proposition_infinitesimal_constant_term_in_asymptotic_expansion_Hermitian_absolute_boundary_conditions}
 Recall the remarks and the notation from Section \ref{Boundary_conditions_and_Poincae_duality}.
 Let $(M,\partial M,\emptyset)$ be a compact Riemannian bordism. Consider $[\Delta,{\mathcal{B}}]^{E,g,h}_{(M,\partial M,\emptyset)}$
 the Hermitian boundary value problem and denote by
 $\Delta_{\mathsf{abs},h}$
 its $L^{2}$-realization. Let here $\STr$ stand for supertrace.
 For $\phi\in\Gamma(M,\End(E))$ we have
\begin{eqnarray}
\label{equation_Proposition_infinitesimal_constant_term_in_asymptotic_expansion_Hermitian_absolute_boundary_conditions_1}
 \scriptstyle\LIM_{t\rightarrow 0} (\STr(\phi
 \exp(-t\Delta_{\mathsf{abs},h})))=
 \int_{M}\Tr(\phi)
 \mathbf{e}(M,g)-(-1)^{m}\int_{\partial M}i^{*}\Tr(\phi)
 \mathbf{e_{b}}(\partial M,g).
\end{eqnarray}
 Moreover,  for $\xi\in\Gamma(M,\End(TM))$ a symmetric endomorphism
 with respect to the metric $g$, and $\mathbf{D}^{*}\xi\in\Gamma(M,\End(\Lambda^{*}T^{*}M))$
 its extension as a derivation on $\Lambda^{*}(T^{*}M)$, set
 \begin{equation}
\label{equation_Proposition_infinitesimal_constant_term_in_asymptotic_expansion_Hermitian_absolute_boundary_conditions_2}
 \Psi:=\mathbf{D}^{*}\xi-\frac{1}{2}\Tr(\xi).
\end{equation}
 If $\tau\in\R$ is taken small enough so that
 $g+\tau g\xi$ is a nondegenerate symmetric metric on $TM$, then
 we have
\begin{eqnarray}
  \nonumber\scriptstyle\LIM_{t\rightarrow 0} \left(\STr\left(-\Psi
	      \exp(-t\Delta_{\mathsf{abs},h})\right)\right)
	    &=&\scriptstyle-2\int_{M}\left.\frac{\partial}{\partial\tau}\right|_{\tau=0
	      }\mathbf{\widetilde{e}}(M,g,g+\tau g\xi)\wedge\omega(\nabla^{E},h)\\
	   &&\label{equation_Proposition_infinitesimal_constant_term_in_asymptotic_expansion_Hermitian_absolute_boundary_conditions_3}
             \scriptstyle+2\int_{\partial M}\left.-\frac{\partial}{\partial\tau}\right|_{\tau=0
	      }\mathbf{\widetilde{e}_{b}}(\partial M,g,g+\tau g\xi)\wedge i^{*}\omega(\nabla^{E},h)\\
	      &&\nonumber\scriptstyle+\mathsf{rank}(E)\int_{\partial M}\left.\frac{\partial}{\partial\tau}\right|_{\tau=0}
	      B(\partial M,g+\tau g\xi),
\end{eqnarray}
 where  $\omega(\nabla^{E},h):=-\frac{1}{2}\Tr(h^{-1}\nabla^{E}h)$ is a real valued closed one-form.
 \end{prop}
     \begin{proof}
     We prove formula (\ref{equation_Proposition_infinitesimal_constant_term_in_asymptotic_expansion_Hermitian_absolute_boundary_conditions_1}).
     First, each $\phi\in\Gamma(M,\End(E))$ can be uniquely written as $\phi=\phi^{\mathsf{re}}+\mathbf{i}\phi^{\mathsf{im}}$ where
     $\phi^{\mathsf{re}},\phi^{\mathsf{im}}$ are selfadjoint elements.
     Thus, it is enough to prove (\ref{equation_Proposition_infinitesimal_constant_term_in_asymptotic_expansion_Hermitian_absolute_boundary_conditions_1}) for $\phi$ selfadjoint.
     First, suppose that $\phi_{u}:=h^{-1}_{u}\frac{\partial h_{u}}{\partial u}\in\Gamma(M,\End(E))$, where $h_{u}$ is a smooth one real parameter family of
     Hermitian forms on $E$ with $h_{0}=h$. Then, (\ref{equation_Proposition_infinitesimal_constant_term_in_asymptotic_expansion_Hermitian_absolute_boundary_conditions_1})
     exactly is the infinitesimal version of Br\"uning and Ma's formulas, see \cite[Theorem 4.6 ]{Bruening-Ma} and
     \cite[expression (5.72)]{Bruening-Ma}. Next, suppose $\phi\in\Gamma(M,\End(E))$ to be an arbitrary selfadjoint element. Then, for $u$ small enough,
     the family $h_{u}:=h+uh\phi$ is a smooth family of Hermitian forms on $E$ and
     $h^{-1}_{u}\frac{\partial h_{u}}{\partial u}=h_{u}^{-1}h\phi$ defines a smooth family of selfadjoint elements in $\Gamma(M,\End(E))$. Therefore, we apply Br\"uning and Ma's formulas for
     $h^{-1}_{0}\left(\left.\frac{\partial h_{u}}{\partial u}\right|_{u=0}\right)=\phi$ so that the proof of (\ref{equation_Proposition_infinitesimal_constant_term_in_asymptotic_expansion_Hermitian_absolute_boundary_conditions_1})
     is complete.
     We now prove (\ref{equation_Proposition_infinitesimal_constant_term_in_asymptotic_expansion_Hermitian_absolute_boundary_conditions_3}).
     Let $g_{u}$ be a smooth family of
     Riemannian metrics on $TM$ with $g_{0}=g$ and denote by $\star_{u}$ the
     Hodge $\star$-operator corresponding to $g_{u}$. First, consider the case where
     $\xi_{u}:=g_{u}^{-1}\frac{\partial g_{u}}{\partial u}\in\Gamma(M;\End(TM))$ so that, by
     (\ref{equation_Proposition_infinitesimal_constant_term_in_asymptotic_expansion_Hermitian_absolute_boundary_conditions_2}), we obtain
     $\Psi_{u}=\mathbf{D}^{*}(g_{u}^{-1}\frac{\partial g_{u}}{\partial u})-\frac{1}{2}\Tr(g_{u}^{-1}\frac{\partial g_{u}}{\partial u})
      =-\star_{u}^{-1}\frac{\partial\star_{u}}{\partial u}$, see \cite[Proposition 4.15]{Bismut-Zhang}, considered as a smooth family in $\Gamma(M,\End(\Lambda^{*}T^{*}M))$. Then,
     (\ref{equation_Proposition_infinitesimal_constant_term_in_asymptotic_expansion_Hermitian_absolute_boundary_conditions_3}) is
     the infinitesimal version of Br\"uning and Ma's formulas, see \cite[Theorem 4.6]{Bruening-Ma} and
     \cite[expressions (5.74) and (5.75)]{Bruening-Ma}. In the general case, take a symmetric $\xi\in\Gamma(M;\End(TM))$. Then, for
     $u$ small enough the formula $g_{u}:=g+ug\xi$ defines a smooth family of nondegenerate metrics on $TM$ and hence
     $g_{u}^{-1}\frac{\partial g_{u}}{\partial u}=g_{u}^{-1}g\xi$ a smooth family of symmetric elements in $\Gamma(M,\End(TM))$. Hence we obtain
     a smooth family of symmetric endomorphisms $-\star_{u}^{-1}\frac{\partial\star_{u}}{\partial u}$ in $\Gamma(M,\End(\Lambda^{*}T^{*}M))$, for which
     we can use again Br\"uning and Ma's formulas. In particular, they must hold for $u=0$ for which we have
     $g^{-1}_{0}(\frac{\partial g_{u}}{\partial u}|_{u=0})=\xi,$ so that
     $
     \Psi_{0}=\mathbf{D}^{*}(\xi)-\frac{1}{2}\Tr(\xi)=
     -\star^{-1}_{0}(\frac{\partial\star_{u}}{\partial u}|_{u=0}).
     $
     That is, (\ref{equation_Proposition_infinitesimal_constant_term_in_asymptotic_expansion_Hermitian_absolute_boundary_conditions_3}) holds.
     \end{proof}

 %The following uses Poincar\'e duality to relate boundary value problems under absolute and relative boundary conditions.
 \begin{lem}
 \label{Lemma_relation_between the infinitesimal_constant_terms_in_asymptotic_expansions_Hermitian_relative_and absolute_boundary_conditions_total}
 Let $\bar{E}^{\prime}$ be the dual of the complex conjugated vector bundle of $E$, endowed with the dual flat connection and dual Hermitian form to those on $E$.
 Consider the compact Riemannian bordisms $(M,\emptyset,\partial M)$ together with its
 dual $(M,\emptyset,\partial M)^{\prime}:=(M,\partial M,\emptyset)$. Let $\Delta_{\mathsf{rel},h}$ be the
 $L^{2}$-realization associated to the Hermitian boundary value problem
 $[\Delta,\mathcal{B}]^{E,g,h}_{(M,\emptyset,\partial M)}$ and ${\Delta^{\prime}_{\mathsf{abs},h^{\prime}}}$ the one associated to
 $[\Delta,\mathcal{B}]^{\bar{E}^{\prime}\otimes\Theta_{M},g,h^{\prime}}_{(M,\emptyset,\partial M)^{\prime}}$.
 % We look at the Hermitian boundary value problem
 %$[\Delta,\mathcal{B}]^{E,g,h}_{(M,\emptyset,\partial M)}$ with $L^{2}$-realization denoted
 %by $\Delta_{\mathsf{rel},h}$ and the its dual Hermitian boundary value problem
 %$[\Delta,\mathcal{B}]^{\bar{E}^{\prime}\otimes\Theta_{M},g,h^{\prime}}_{(M,\emptyset,\partial M)^{\prime}}$ with corresponding
 %$L^{2}$-realization
 %${\Delta^{\prime}_{\mathsf{abs},h^{\prime}}}$.
 If $\phi$, $\xi$ and $\Psi$ are as in Proposition
 \ref{Proposition_infinitesimal_constant_term_in_asymptotic_expansion_Hermitian_absolute_boundary_conditions}, then
 \begin{equation}
 \label{Lemma_relation_between the infinitesimal_constant_terms_in_asymptotic_expansions_Hermitian_case_1_relative_and absolute_boundary_conditions}
    { }\LIM_{t\rightarrow 0} \left(\STr\left(\phi
    \exp(-t\Delta_{\mathsf{rel},h})\right)\right)=
    (-1)^{m}\LIM_{t\rightarrow 0} \left(\STr\left(\phi^{*}
    \exp{-t{\Delta^{\prime}_{\mathsf{abs},h^{\prime}}}}\right)\right),
 \end{equation}
 where $\phi^{*}:=h\phi h^{-1}$,
 and
 \begin{equation}
 \label{Lemma_relation_between the infinitesimal_constant_terms_in_asymptotic_expansions_Hermitian_case_2_relative_and absolute_boundary_conditions}
     { }\LIM_{t\rightarrow 0}\STr\left(\Psi
    \exp(-t\Delta_{\mathsf{rel},h})\right)=
    (-1)^{m+1}\LIM_{t\rightarrow 0}\STr\left(\Psi
    \exp(-t{\Delta^{\prime}_{\mathsf{abs},h^{\prime}}})\right).
 \end{equation}
\end{lem}
\begin{proof}
 Consider $h\in\Omega^{0}(M;\End(E,\bar{E}^{\prime}))$ the complex vector bundle isomorphism between
 $E$ and $\bar{E}^{\prime}$ provided by the Hermitian metric on $E$ (see for instance \cite[page 286]{Bott-Tu}),
 and its covariant derivative $\nabla^{E} h\in\Omega^{1}(M;\End(E,\bar{E}^{\prime}))$ computed by using the induced connection on $\End(E,\bar{E}^{\prime})$.
 With the Hermitian metric on $E$ and
 the Riemannian metric
 on $M$, we have a \emph{complex linear} isomorphism
 $\star_{h}:=\star\otimes h:\Omega(M;E)\rightarrow\Omega(M;\bar{E}^{\prime}\otimes\Theta_{M})$, which is used to define
 $$
 d^{*}_{E,g,h}:=(-1)^{q}\star_{h}^{-1}d_{\bar{E}^{\prime}\otimes\Theta_{M}}\star_{h}:
 \Omega^{q}(M;E)\rightarrow\Omega^{q-1}(M;E);
 $$
 being the formal adjoint to $\mathop{d_{E}}$ with respect to the Hermitian product on $\Omega(M;E)$.
 Remark here that
 the formula
 $$
 d_{\bar{E}^{\prime}\otimes\Theta_{M}}d^{*}_{\bar{E}^{\prime}\otimes\Theta_{M},g,h^{\prime}}\star_{h}=
 \star_{h}d^{*}_{E,g,h}\mathop{d_{E}}
 $$
 holds and therefore $$\star_{h}\Delta_{E,g,h}=
 \Delta_{\bar{E}^{\prime}\otimes\Theta_{M},g,h^{\prime}}\star_{h}.$$
  As in Section \ref{Boundary_conditions_and_Poincae_duality},
 the operator $\star_{h}$ intertwines $E$-valued forms satisfying
 relative (resp. absolute) boundary conditions with $\bar{E}^{\prime}$-valued forms
 satisfying absolute (resp. relative) boundary conditions. That is,
 \begin{equation}
 \label{equation_Laplacian_intertwining_conjugate_dual}
 \Delta_{\mathsf{rel},h}
 =\star_{h}^{-1}
 {\Delta^{\prime}_{\mathsf{abs},h^{\prime}}}\star_{h}
 \end{equation}
  and therefore
  $
  \phi \exp(-t\Delta_{\mathsf{rel},h})=
  \star_{h}^{-1}\phi^{*}\exp(-t
  {\Delta^{\prime}_{\mathsf{abs},h^{\prime}}})\star_{h},
  $
  where  $\phi^{*}:=h\phi h^{\prime}$.
 Thus, since the supertrace vanishes on supercommutators of graded complex-linear operators and
 the degree of $\star_{h,q}$ is $m-q$, we obtain the formula
  $$
  \STr(\phi \exp(-t\Delta_{\mathsf{rel},h}))=
  (-1)^{m}\STr(\phi^{*} \exp(-t
  {\Delta^{\prime}_{\mathsf{abs},h^{\prime}}}))
  $$
  and hence
 (\ref{Lemma_relation_between the infinitesimal_constant_terms_in_asymptotic_expansions_Hermitian_case_1_relative_and absolute_boundary_conditions}).
 We now turn to formula (\ref{Lemma_relation_between the infinitesimal_constant_terms_in_asymptotic_expansions_Hermitian_case_2_relative_and absolute_boundary_conditions}). First, remark that
 \begin{equation}
 \label{equation_2_in_the_proof_of_super_trace_relative_absolute_and_psi_on_the_exterior_algebra}
\begin{array}{lll}
 \star_{q}\left(\mathbf{D}^{*}\xi-\frac{1}{2}\Tr(\xi)\right)
 \star_{q}^{-1}=-\mathbf{D}^{*}\xi+\frac{1}{2}\Tr(\xi).
\end{array}
 \end{equation}
 We prove (\ref{equation_2_in_the_proof_of_super_trace_relative_absolute_and_psi_on_the_exterior_algebra}),
 by pointwise computing $\star_{q}\mathbf{D}^{*}\xi\star_{q}^{-1}$. %, since $\Tr(\xi)$ commutes with $\star$.
 %We \footnote{Pointwise!} evaluate this operator on a generic element of a particular choice of an orthonormal local frame for
 %$\Lambda^{*}T^{*}M$. We fix this frame as follows.
 Since $\xi$ is a symmetric complex endomorphism of $T_{x}M$, we may choose
 an orthonormal frame $\{e_{i}\}_{1}^{m}$ such that $\xi e_{i}=\lambda_{i} e_{i}$. Then, for
 $\{e^{i_{1}}\wedge\cdots\wedge e^{i_{q}}\}_{1\leq i_{1}<\cdots < i_{q}\leq m}$
 a positive definite oriented frame for $\Lambda^{q}T^{*}_{x}M$, the Hodge $\star$-operator is given by
 $
 \star_{q}\left(e^{i_{1}}\wedge\cdots\wedge e^{i_{q}}\right)=e^{j_{1}}\wedge\cdots\wedge e^{j_{m-q}}\in \Lambda^{m-q}T^{*}_{x}M,
 $
 where the ordered indices  $(j_{1},\ldots,j_{m-q}):=(1,\ldots,\widehat{i_{1}},\ldots,\widehat{i_{q}},\ldots,m)$ with
 $1\leq j_{1}<\ldots<j_{m-q}\leq m$, are obtained as the unique possible choice of ordered indices complementary to
 $\leq i_{1}<\cdots < i_{q}$. Therefore
\begin{eqnarray}
  \nonumber\scriptstyle\star_{q}\mathbf{D}^{*}\xi
 \star_{q}^{-1}\left(e^{j_{1}}\wedge\cdots\wedge e^{j_{m-q}}\right)
 &=&\scriptstyle\star_{q}\mathbf{D}^{*}\xi\left(e^{i_{1}}\wedge\cdots\wedge e^{i_{q}}\right)\\
 &=&\nonumber \scriptstyle\star_{q}\sum_{l=1}^{q}\left(e^{i_{1}}\wedge\cdots\wedge\xi(e^{i_{l}})\wedge\cdots\wedge e^{i_{q}}\right)\\
 &=&\nonumber \scriptstyle\star_{q}\sum_{l=1}^{q}\lambda_{i_{l}}\left(e^{i_{1}}\wedge\cdots\wedge e^{i_{l}}\wedge\cdots\wedge e^{i_{q}}\right)\\
 &=&\nonumber \scriptstyle\sum_{l=1}^{q}\lambda_{i_{l}}\left(e^{j_{1}}\wedge\cdots\wedge e^{j_{m-q}}\right)\\
 &=&\nonumber \scriptstyle\sum_{l=1}^{m}\lambda_{i_{l}}\left(e^{j_{1}}\wedge\cdots\wedge e^{j_{m-q}}\right)				-
 \sum_{l=1}^{m-q}\lambda_{j_{l}}\left(e^{j_{1}}\wedge\cdots\wedge e^{j_{m-q}}\right)\\
 &=&\nonumber\scriptstyle\sum_{l=1}^{m}\lambda_{i_{l}}\left(e^{j_{1}}\wedge\cdots\wedge e^{j_{m-q}}\right)
-\sum_{l=1}^{m-q}\left(e^{j_{1}}\wedge\cdots\wedge \lambda_{j_{l}} e^{j_{l}} \wedge\cdots\wedge e^{j_{m-q}}\right)\\
 &=&\nonumber\scriptstyle\left(\Tr{\xi}-\mathbf{D}^{*}\xi\right)\left(e^{j_{1}}\wedge\cdots\wedge e^{j_{m-q}}\right)
\end{eqnarray}
 and we obtain (\ref{equation_2_in_the_proof_of_super_trace_relative_absolute_and_psi_on_the_exterior_algebra}), which in turn allows us to conclude
\begin{eqnarray}
 \nonumber\scriptstyle
 \Psi
 \left(\star_{q}\otimes h\right)^{-1}&=&\scriptstyle
 \left(\left(\mathbf{D}^{*}\xi-\frac{1}{2}\Tr(\xi)\right)\otimes 1\right)\left(\star_{q}\otimes h\right)^{-1}\\
 \label{equation_1_in_the_proof_of_super_trace_relative_absolute_and_psi_on_the_exterior_algebra}
 &=& \scriptstyle\left(\star_{q}\otimes h\right)^{-1}\left(\left(\star_{q}\left(\mathbf{D}^{*}\xi-\frac{1}{2}\Tr(\xi)\right)
 \star_{q}^{-1}\right)\otimes 1\right)\\
 &=&\nonumber\scriptstyle -\left(\star_{q}\otimes h\right)^{-1}\left(\left(\mathbf{D}^{*}\xi-\frac{1}{2}\Tr(\xi)\right)\otimes 1\right)\\
 &=&\nonumber\scriptstyle -\left(\star_{q}\otimes h\right)^{-1}\Psi.
\end{eqnarray}
 Finally, we use (\ref{equation_1_in_the_proof_of_super_trace_relative_absolute_and_psi_on_the_exterior_algebra}) to pass
 to the complex conjugated; hence with (\ref{equation_Laplacian_intertwining_conjugate_dual})
 and
 duality between these boundary value problems we obtain
 $$
  \Psi \exp\left(-t\Delta_{\mathsf{rel},h}\right)
   =\Psi \star_{h}^{-1}
   \exp\left(-t{\Delta^{\prime}_{\mathsf{abs},h^{\prime}}}\right)
   \star_{h}\\
   =- \star_{h}^{-1}\Psi
   \exp(-t{\Delta^{\prime}_{\mathsf{abs},h^{\prime}}})
   \star_{h}\\
 $$
thus, as for
(\ref{Lemma_relation_between the infinitesimal_constant_terms_in_asymptotic_expansions_Hermitian_case_1_relative_and absolute_boundary_conditions}),
we have
 $$
  \STr(\Psi \exp(-t\Delta_{\mathsf{rel},h}))
  =-(-1)^{m}\STr(\Psi \exp(-t
  {\Delta^{\prime}_{\mathsf{abs},h^{\prime}}}))
 $$
\end{proof}

\begin{prop}
\label{Proposition_infinitesimal_constant_term_in_asymptotic_expansion_Hermitian_case_relative_boundary_conditions}
 For the  Riemannian bordism $(M,\emptyset,\partial M)$, consider the Hermitian boundary value problem
 $[\Delta,\mathcal{B}]^{E,g,h}_{(M,\emptyset,\partial M)}$ with its $L^{2}$-realization denoted by
 $\Delta_{\mathsf{rel},h}$.
 If $\phi$, $\xi$ and $\Psi$ are as in Proposition
 \ref{Proposition_infinitesimal_constant_term_in_asymptotic_expansion_Hermitian_absolute_boundary_conditions}, then
\begin{eqnarray}
  \nonumber\scriptstyle\LIM_{t\rightarrow 0} \left(\STr\left(\phi
  \exp(-t\Delta_{\mathsf{rel},h})\right)\right)=
  \int_{M}\Tr(\phi)
  \mathbf{e}(M,g)-
  \int_{\partial M}i^{*}\Tr(\phi)
  \mathbf{e_{b}}(\partial M,g).
\end{eqnarray}
 and
\begin{eqnarray}
 \nonumber\scriptstyle\LIM_{t\rightarrow 0} \left(\STr\left(-\Psi
 \exp(-t\Delta_{\mathsf{rel},h})\right)\right)
&=&\scriptstyle-2\int_{M}\left.\frac{\partial}{\partial\tau}\right|_{\tau=0
 }\mathbf{\widetilde{e}}(M,g,g+\tau g\xi)
 \wedge \omega(\nabla^{E},h)\\
&&\nonumber\scriptstyle+2(-1)^{m+1}\int_{\partial M}\left.\frac{\partial}{\partial\tau}\right|_{\tau=0
 }\mathbf{\widetilde{e}_{b}}(\partial M,g,g+\tau g\xi)
 \wedge i^{*}\omega(\nabla^{E},h)\\
&&\nonumber\scriptstyle
 +(-1)^{m+1}\mathsf{rank}(E)\int_{\partial M}\left.\frac{\partial}{\partial\tau}\right|_{\tau=0}
B(\partial M,g+\tau g\xi)\nonumber.
\end{eqnarray}
\end{prop}
\begin{proof}
A form $w\in\Omega^{*}(M;E)$ satisfies relative boundary conditions if and only if the smooth form
$\star_{h}w\in\Omega^{m-*}(M;\bar{E}^{\prime}\otimes\Theta_{M})$ satisfies absolute boundary conditions
on $\partial M$.
Hence, the first formula in the statement follows from
formula (\ref{Lemma_relation_between the infinitesimal_constant_terms_in_asymptotic_expansions_Hermitian_case_1_relative_and absolute_boundary_conditions}) in
Lemma
\ref{Lemma_relation_between the infinitesimal_constant_terms_in_asymptotic_expansions_Hermitian_relative_and absolute_boundary_conditions_total},
and the results from Br\"uning and Ma for the Hermitian Laplacian stated
in Proposition \ref{Proposition_infinitesimal_constant_term_in_asymptotic_expansion_Hermitian_absolute_boundary_conditions}. The second
formula follows from Lemma
formula (\ref{Lemma_relation_between the infinitesimal_constant_terms_in_asymptotic_expansions_Hermitian_case_2_relative_and absolute_boundary_conditions}) in
\ref{Lemma_relation_between the infinitesimal_constant_terms_in_asymptotic_expansions_Hermitian_relative_and absolute_boundary_conditions_total},
Proposition \ref{Proposition_infinitesimal_constant_term_in_asymptotic_expansion_Hermitian_absolute_boundary_conditions}
and
$\omega(\nabla^{E},h)=-\omega(\nabla^{E^{\prime}},h^{\prime})$, see for instance
\cite[Section 2.4]{Burghelea-Haller}.
\end{proof}

\begin{lem}
\label{Lemma_computation_of_Heat_kernel_asymptotics_of_total_boundary_problem_Reduced_to_small_parts}
For $(M,\partial M,\emptyset)$, $(M,\emptyset,\partial M)$ and $(M,\partial_{+}M,\partial_{-}M)$ let us
consider
 $[\Delta,\mathcal{B}]^{E,g,h}_{(M,\partial M,\emptyset)}$, $[\Delta,\mathcal{B}]^{E,g,h}_{(M,\emptyset,\partial M)}$
and $[\Delta,\mathcal{B}]^{E,g,h}_{(M,\partial_{+}M,\partial_{-}M)}$ the corresponding
Hermitian boundary value problems, together with their $L^{2}$-realizations
$\Delta_{\mathsf{abs},h}$, $\Delta_{\mathsf{rel},h}$ and $\Delta_{\mathcal{B},h}$, respectively.
Let $\psi_{\pm}\in\Gamma(M;\End(\Lambda^{*}(T^{*}M)\otimes E))$ be chosen
in such a way that $\mathsf{supp}(\psi_{\pm})\cap\partial_{\mp}M=\emptyset$, then
$$\begin{array}{l}
 { }\LIM_{t\rightarrow 0}(\STr(\psi_{+}
\exp(-t\Delta_{\mathcal{B},h})))
=\LIM_{t\rightarrow 0}(\STr(\psi_{+}
\exp(-t\Delta_{\mathsf{abs},h}))),\\ \\
{ }\LIM_{t\rightarrow 0}(\STr(\psi_{-}
\exp(-t\Delta_{\mathcal{B},h})))
=\LIM_{t\rightarrow 0}(\STr(\psi_{-}
\exp(-t\Delta_{\mathsf{rel},h}))).
\end{array}
$$
\end{lem}
\begin{proof}
This is a immediate consequence of $\partial_{+}M$ and
$\partial_{-}M$ being mutually disjoint and that the coefficients in the heat kernel asymptotic expansion are computable
as universal polynomials in terms of finite order derivatives of the symbols expressed in local coordinates around each point of $M$, see
Section \ref{subsection_Heat trace asymptotics for the heat kernel of a Laplace type operator under local boundary conditions}.
\end{proof}

\begin{thm}
\label{Theorem infinitesimal_constant_terms_in_asymptotic_expansions_Hermitian_case_relative_and absolute_boundary_conditions}
For $(M,\partial_{+}M,\partial_{-}M)$, consider the Hermitian boundary value problem
$[\Delta,\mathcal{B}]^{E,g,h}_{(M,\partial_{+}M,\partial_{-}M)}$
with its corresponding $L^{2}$-realization $\Delta_{\mathcal{B},h}$. If
$\phi$, $\xi$ and $\Psi$ are as in Proposition
\ref{Proposition_infinitesimal_constant_term_in_asymptotic_expansion_Hermitian_absolute_boundary_conditions}, then
\begin{eqnarray}
  \nonumber\scriptstyle
  \LIM_{t\rightarrow 0} \left(\STr\left(\phi
  \exp(-t\Delta_{\mathcal{B},h})\right)\right)&=&\scriptstyle\int_{M}\Tr(\phi)\mathbf{e}(M,g)
  +(-1)^{m-1}\int_{\partial_{+} M}\Tr(\phi) i^{*}_{+}\mathbf{e_{b}}(\partial M,g)\\
 &&\nonumber\scriptstyle-\int_{\partial_{-} M}\Tr(\phi)
  i^{*}_{-}\mathbf{e_{b}}(\partial M,g).
\end{eqnarray}
and
\begin{eqnarray}
  \nonumber\scriptstyle
    { }\LIM_{t\rightarrow 0} \left(\STr\left(-\Psi
   \exp(-t\Delta_{\mathcal{B},h})\right)\right)&=&\scriptstyle-2\int_{M}
    \left.\frac{\partial}{\partial\tau}\right|_{\tau=0
    }\mathbf{\widetilde{e}}(M,g,g+\tau g\xi)\wedge
   \omega(\nabla^{E},h)\\
&&\nonumber\scriptstyle-2\int_{\partial_{+}M}
    \left.\frac{\partial}{\partial\tau}\right|_{\tau=0
    }i^{*}_{+}\mathbf{\widetilde{e}_{b}}(\partial M,g,g+\tau g\xi)\wedge
    \omega(\nabla^{E},h)\\
			&&\nonumber\scriptstyle+\mathsf{rank}(E)\int_{\partial_{+} M}\left.\frac{\partial}{\partial\tau}\right|_{\tau=0}
			i^{*}_{+}B(\partial M,g+\tau g\xi)\\
&&\nonumber\scriptstyle-2(-1)^{m}\int_{\partial_{-}M}
   \left.\frac{\partial}{\partial\tau}\right|_{\tau=0
   }i^{*}_{-}\mathbf{\widetilde{e}_{b}}(\partial M,g,g+\tau g\xi)\wedge
  \omega(\nabla^{E},h)\\
			&&\nonumber\scriptstyle+(-1)^{m+1}\mathsf{rank}(E)\int_{\partial_{-} M}\left.\frac{\partial}{\partial\tau}\right|_{\tau=0}
			i^{*}_{-}B(\partial M,g+\tau g\xi).\nonumber
\end{eqnarray}
\end{thm}
\begin{proof}
This follows from
Proposition \ref{Proposition_infinitesimal_constant_term_in_asymptotic_expansion_Hermitian_absolute_boundary_conditions} (Br\"uning and Ma),
Proposition
\ref{Proposition_infinitesimal_constant_term_in_asymptotic_expansion_Hermitian_case_relative_boundary_conditions}
and Lemma
\ref{Lemma_computation_of_Heat_kernel_asymptotics_of_total_boundary_problem_Reduced_to_small_parts}.
More recently, Br\"uning and Ma gave also a proof of this
statement, see  \cite[Theorem 3.2]{Bruening-Ma2}, based on the methods developed in \cite{Bruening-Ma}.
\end{proof}

\subsection{Involutions, bilinear and Hermitian forms}
\label{Section_Involutions, bilinear and Hermitian forms}
We fix a Hermitian structure compatible with the bilinear one as follows.
Since $E$ is endowed with a bilinear
form $b$,
there exists an anti-linear involution $\nu$ on $E$
satisfying
\begin{equation}
 \label{conditions_for_involution}
  \overline{b(\nu e_{1},\nu e_{2})}=b(e_{1},e_{2})\text{
 and }b(\nu e,e)>0\text{ for all  }e_{1},e_{2}, e\in E\text{ with }e\not=0,
\end{equation}
see for instance the proof of
\cite[Theorem 5.10]{Burghelea-Haller}. In this way, we obtain a (positive definite) Hermitian form on $E$ given by
\begin{equation}
 \label{compatible_hermitian_form}
  h(e_{1},e_{2}):=b(e_{1},\nu e_{2}).
\end{equation}
Remark that $\nabla^{E}\nu=0$ is not required so that
$$h^{-1}(\left.\nabla^{E}\right. h)=
\nu^{-1}\left(b^{-1}(\left.\nabla^{E}\right. b)\right)\nu+\nu^{-1}(\left.\nabla^{E}\right.\nu).$$
Therefore, this yields
a Hermitian form on $\Omega(M;E)$ compatible with $\beta_{g,b}$ in the sense that
$ \ll  v, w\gg_{g,h}=\beta_{g,b}(v, \nu w).$
 for $v,w\in\Omega(M;E)$.
 In \cite{Su-Zhang08} and \cite{Su1}, given a bilinear form $b$,
 this involution has been exploited to
 study the bilinear Laplacian in terms of the Hermitian one associated to the compatible Hermitian form in (\ref{compatible_hermitian_form}), in
 both cases with and without boundary. However, our approach is a little different since
 we do not use a Hermitian form globally compatible with $\beta_{g,b}$ on $\Omega(M;E)$, but instead a local compatibility only,
 see section
 \ref{subsection_Constant term in the heat trace asymptotic expansions for boundary problems} below.

 We now study the situation where $\nu$ is parallel with respect to $\nabla^{E}$.

\begin{lem}
\label{lemma_parallel_involution_implies_d_adjoint_equals_d_transpose}
 Let us consider $(M,\partial_{+}M,\partial_{-}M)$ the compact Riemannian bordism together with
 the complex flat vector bundle $E$ as above. Suppose $E$ admits a nondegenerate symmetric bilinear form.
 Moreover, suppose there exists a complex anti-linear involution $\nu$ on $E$,
 satisfying the conditions in (\ref{conditions_for_involution}) and $\nabla^{E}\nu=0$.
 Let $h$ be the (positive definite) Hermitian form on $E$ compatible with $b$ defined by
(\ref{compatible_hermitian_form}). Then,
 $\Delta_{E,g,b}=\Delta_{E,g,h}$ and $\mathcal{B}_{E,g,b}=\mathcal{B}_{E,g,h}.$
 \end{lem}
\begin{proof}
 Consider $\ll  \cdot, \cdot\gg_{g,h}$ the Hermitian product on $\Omega(M;E)$, compatible with the bilinear form, and
 $d^{*}_{E,g,h}$, the formal adjoint to $\mathop{d_{E}}$ with respect to this product, which
 in terms of the Hodge $\star$-operator can be written up to a sign as
 $d^{*}_{E,g,h}=\pm\star^{-1}_{h}\mathop{d_{E}}\star_{h}$. Remark that
 $\nabla^{E}\nu=0$ implies that $\mathop{d_{E}}\nu=\nu \mathop{d_{E}}$; hence, with $\star_{h}=\nu\circ\star_{b}$, we have
 \begin{equation}
 \label{equation_d_star_h_equals_d_star_b}
 d^{*}_{E,g,h}=\pm\star^{-1}_{h}\mathop{d_{E}}\star_{h}%=(\pm)\star^{-1}_{b}\nu^{-1} \mathop{d_{E}}\nu \star_{b}
              =\pm\star^{-1}_{b}\nu^{-1} \mathop{d_{E}}\nu \star_{b}
	      =\pm\star^{-1}_{b} \mathop{d_{E}} \star_{b}=\mathop{d^{\sharp}_{E,g,b}},
 \end{equation}
 and therefore the Hermitian and bilinear Laplacians coincide. We turn to the assertion
 for the corresponding boundary operators.
 On the one hand, the assertion is clear for
 ${\mathcal{B}_{-}}_{E,g,b}={\mathcal{B}_{-}}_{E,g,h}$, because of
 (\ref{equation_d_star_h_equals_d_star_b}) and (\ref{definition_relative_Boundary_operator_on_smooth forms}). On the other hand,
 for a form $v\in\Omega^{p}(M;E)$ and
 $\iota_{\varsigma_{\mathsf{in}}}$, the interior product with respect to the dual form corresponding to
 $\varsigma_{\mathsf{in}}$, the identity
 $\star^{\partial M}_{b} i^{*}\iota_{\varsigma_{\mathsf{in}}} v=i^{*}\star^{M}_{b} v$ holds; therefore
 the operator specifying absolute boundary can be written, independently of the Hermitian or bilinear forms, as
 $
 {\mathcal{B}_{+}}^{p}_{E,g,b}v=
  (i^{*}_{+}\iota_{\varsigma_{\mathsf{in}}}v,(-1)^{p+1} i^{*}_{+}\iota_{\varsigma_{\mathsf{in}}}(\mathop{d_{E}}v))=
 {\mathcal{B}_{+}}^{p}_{E,g,h}v.
 $
 That finishes the proof.
\end{proof}

\begin{lem}
\label{Lemma_existing_neighbourhood_and_parallel_involution}
Let $(M,g)$ be a compact Riemannian manifold and $E$ a flat complex vector bundle over $M$.
Assume $E$ is endowed with
a fiberwise nondegenerate symmetric bilinear form $b$. For each $x\in M$ there exists an open neighborhood $U$ of $x$ in $M$, a parallel anti-linear involution
$\nu$ on $E|_{U}$ and a symmetric bilinear form $\widetilde{b}$ on $E$ such that, for $z\in\C$, the family of fiberwise symmetric bilinear forms
\begin{equation}
\label{family_b_z}
b_{z}:=b+z\widetilde{b},
\end{equation}
has the following properties.
\begin{itemize}
\item[(i)] $b_{z}$ is fiberwise nondegenerate for all $z\in\C$ with $|z|\leq \sqrt{2}$,
\item[(ii)] $\overline{b_{s-\mathbf{i}}(\nu e_{1},\nu e_{2})}=b_{s-\mathbf{i}}(e_{1}, e_{2})$, for all $s\in \R$ and $e_{i}\in E|_{U},$
\item[(iii)] $b_{s-\mathbf{i}}(e,\nu e)>0$ for all $s\in\R$, $|s|\leq 1$ and $0\not=e\in E|_{U}.$
\end{itemize}
\end{lem}
\begin{proof}
Since flat vector bundles are locally trivial, there exists a neighborhood $V$ of $x$ and a parallel
complex anti-linear involution $\nu$ on $E|_{V}$. Moreover, since $b$ is nondegenerate and $\nu$ an involution, we can
assume without loss of generality that $\nu$ can be chosen to be compatible with $b$ at the fiber $E_{x}$ over $x$, such that
$$
b_{x}(\nu e_{1},\nu e_{2})=\overline{b_{x}(e_{1},e_{2})}\quad\text{for all}\quad e_{i}\in E_{x}
$$
 and
$$
b_{x}(\nu e, e)>0\quad\text{for all}\quad 0\not=e\in E_{x}.
$$
Consider
$$
\begin{array}{c}
b^{\mathsf{Re}}(e_{1},e_{2}):=\frac{1}{2}\left(b(e_{1},e_{2})+\overline{b(\nu e_{1},\nu e_{2})}\right),\\
b^{\mathsf{Im}}(e_{1},e_{2}):=\frac{1}{2\mathbf{i}}\left(b(e_{1},e_{2})-\overline{b(\nu e_{1},\nu e_{2})}\right),
\end{array}
$$
as symmetric bilinear forms on $E|_{V}$.
In particular, note that by construction
\begin{equation}
\label{proof_equation_a}
b|_{V}=b^{\mathsf{Re}}+\mathbf{i}b^{\mathsf{Im}}\quad\text{with}\quad b^{\mathsf{Im}}|_{E_{x}}=0,
\end{equation}
\begin{equation}
 \label{proof_equation_b}
\overline{b^{\mathsf{Re}}(\nu e_{1},\nu e_{2})}=b^{\mathsf{Re}}(e_{1},e_{2})\quad\text{and}\quad
\overline{b^{\mathsf{Im}}(\nu e_{1},\nu e_{2})}=b^{\mathsf{Im}}(e_{1},e_{2}),
\end{equation}
for all $e_{i}\in E|_{V}$.
Now, choose an open neighborhood $U\subset V$ of $x$ and a compactly supported smooth function $\lambda:V\rightarrow[0,1]$ such that
$\lambda|_{U}=1$. Thus, by extending $\lambda$ by zero to $M$, we set
\begin{equation}
 \label{proof_equation_b_tilde}
\widetilde{b}:=\lambda b^{\mathsf{Im}},
\end{equation}
as a globally defined symmetric bilinear form on $E$. Using
$$
b_{s-\mathbf{i}}|_{U}=\left(b+(s-\mathbf{i})\widetilde{b}\right)|_{U}=
b|_{U}+(s-\mathbf{i})b^{\mathsf{Im}}|_{U}=b^{\mathsf{Re}}|_{U}+sb^{\mathsf{Im}}|_{U}
$$
and (\ref{proof_equation_b}) we immediately obtain (ii). In turn, (ii) implies $$\overline{b_{s-\mathbf{i}}(\nu e, e)}=b_{s-\mathbf{i}}(\nu e, e)$$ and hence $b_{s-\mathbf{i}}(\nu e, e)$ is real for all $s\in\R$ and $e\in E|_{U}$.
Finally, by the formula (\ref{family_b_z}) defining $b_{z}$ at $x$, we have $b^{\mathsf{Im}}|_{x}=0$ and therefore
\begin{itemize}
 \item $b_{z}|_{x}$ is nondegenerate,
 \item $b_{s-\mathbf{i}}|_{x}(\nu e,e)=b|_{x}(\nu e,e)>0$ for all $0\not= e\in E_{x}$,
\end{itemize}
from which (i) (resp. (iii)) follows by taking $|z|\leq\sqrt{2}$ (resp. $|s|\leq 1$) and then choosing the support of $\lambda$ small enough around $x$.
\end{proof}
The following Proposition provides the key argument in
the proof of Theorem
\ref{Theorem infinitesimal_constant_terms_in_asymptotic_expansions_bilinear_case_relative_and absolute_boundary_conditions} below.
\begin{prop}
\label{Proposition_family_of_bilinear_and_Hermitian_forms}
Let $[\Delta,\mathcal{B}]^{E,g,b}_{(M,\partial_{+}M,\partial_{-}M)}$ be the bilinear boundary
value problem under absolute and relative boundary conditions on $(M,\partial_{+}M,\partial_{-}M)$.
Then, for each $x\in M$, there exist $\{b_{z}\}_{z\in\C}$ a family of \textit{fiberwise symmetric
bilinear forms} on $E$,  and $\{h_{s}\}_{s\in\R}$ a family of fiberwise \textit{sesquilinear
Hermitian forms} on $E$ such that
\begin{enumerate}
\item[(i)] $b_{z}$ is fiberwise nondegenerate for all $z\in\C$ such that $|z|\leq\sqrt{2}$.
\item[(ii)]$h_{s}$ is fiberwise positive definite Hermitian form for $s\in\R$ with $|s|\leq 1$.
\item[(iii)] For each $s\in\R$ with $|s|\leq 1$,
             consider
	     $[{\Delta,\Omega_{\mathcal{B}}}]^{E,g,h_{s}}_{(M,\partial_{+}M,\partial_{-}M)}$
             the corresponding Hermitian boundary value problem. Then,
  there exists a neighborhood $U$ of $x$ such that
$$
\begin{array}{lcr}
\Delta_{E,g,b_{s-\mathbf{i}}}|_{U}=\Delta_{E,g,h_{s}}|_{U}&\text{ and }&
\mathcal{B}_{E,g,b_{s-\mathbf{i}}}|_{U}=\mathcal{B}_{E,g,h_{s}}|_{U}.
\end{array}
$$
\end{enumerate}
\end{prop}
\begin{proof}
By Lemma  \ref{Lemma_existing_neighbourhood_and_parallel_involution}.(i), for each $x\in M$, there exists a globally defined
fiberwise symmetric bilinear form $\widetilde{b}$ on $E$ such that the formula $b_{z}:=b+z\widetilde{b}$ in
(\ref{family_b_z})
%the formula
%\begin{equation}
% \label{proof_equation_anomalyformula_1}b_{z}:=b+z\widetilde{b},\quad\text{for}\quad z\in\C,
%\end{equation}
defines a family of fiberwise nondegenerate symmetric bilinear forms on $E$, satisfying
the required property in (i).
In addition, we know that for each $x\in M$,
there exist an open neighborhood $V$ of $x$ and a
parallel complex anti-linear involution $\nu$ on $E|_{V}$.
By Lemma \ref{Lemma_existing_neighbourhood_and_parallel_involution}.(i)-(ii), we also know that
we can find  $U\subset V$ a small enough open neighborhood of $x$, such that $b_{s-\mathbf{i}}$ satisfies the conditions (i) and (ii) on $E|_{U}$, for $|s|\leq 1$.
Hence, by using the formula in (\ref{compatible_hermitian_form}), we obtain a fiberwise positive definite Hermitian form compatible with $b_{s-\mathbf{i}}$ on $E|_{U}$ given by
$h_{s}^{U}(e_{1},e_{2}):=b_{s-\mathbf{i}}(\nu e_{1},e_{2}).$
Now we extend $h_{s}^{U}$ to a (positive definite) Hermitian form on $E$ as follows. We take $h^{\prime}$
any arbitrary Hermitian form on $E$ and consider the finite
open covering $\{U^{\prime}_{0},U^{\prime}_{1}\ldots,U^{\prime}_{N}\}$ of $M$, with $U^{\prime}_{0}:=U$, together with a subordinate partition of unity $\{f_{j}\}_{U^{\prime}_{j}}$. If
$h_{j}^{\prime}:=h^{\prime}|_{U_{j}}$, then $h_{s}:=f_{0}h_{s}^{U}+\sum_{j=1}^{N}f_{j}h_{j}^{\prime}$ globally defines a fiberwise positive definite Hermitian form on $E$,
as the space of Hermitian forms on $E$ is a convex space. This proves (ii). Then,
(iii) follows from Lemma \ref{lemma_parallel_involution_implies_d_adjoint_equals_d_transpose}.
\end{proof}

\subsection{Heat trace asymptotics for bilinear boundary value problems}
\label{subsection_Constant term in the heat trace asymptotic expansions for boundary problems}
\begin{lem}
\label{Lemma_holomorphic_dependance_of_a_m_on_a_parameter_u}
Let $O$ be an open connected subset in $\C$ and $\{z\mapsto b_{z}\}_{z\in U}$ a holomorphic
family of fiberwise nondegenerate symmetric bilinear
forms on $E$.
For the bordism $(M,\partial_{+}M,\partial_{-}M)$ consider
$%\left[\Delta,\Omega_{\mathcal{B}}\right]_{z}:=
\{\left[\Delta,\Omega_{\mathcal{B}}\right]^{E,g,b_{z}}_{(M,\partial_{+}M,\partial_{-}M)}\}_{z\in O},$ the family
of boundary value problems
corresponding to bilinear Laplacians under absolute/relative boundary conditions, together with their $L^{2}$-realizations denoted by
$\Delta_{\mathcal{B},b_{z}}$.
Then, for each $\psi\in\End(\Lambda T^{*}M\otimes E)$,
 the map
$$
z\mapsto\LIM_{t\rightarrow 0}\left(\STr\left(\psi \exp(-t\Delta_{\mathcal{B},b_{z}})\right)\right)
$$
is holomorphic on $O$.
\end{lem}

\begin{proof}
By compactness, we may
assume without loss of generality that $\psi$  %and a
%finite number of its covariant derivatives are
is compactly supported in the interior of a sufficiently small open set $U$ in $M$.
Remark that
the function $z\mapsto b_{z}^{-1}$ is holomorphic, since $z\mapsto b_{z}$ is a holomorphic family of fiberwise nondegenerate bilinear forms in $z\in O$.
Then, as it can directly be checked by
construction of the bilinear Laplacian in (\ref{formula_definition_Laplacian}) and the boundary operators in (\ref{definition_Boundary_operator_on_smooth forms}),
the assignments
$z\mapsto\Delta_{E,g,b_{z}}$ and $z\mapsto\mathcal{B}_{E,g,b_{z}}$ respectively define holomorphic functions in $z\in O$.
%since the mappings $z\mapsto \star\otimes b_{z}$, $z\mapsto \nabla^{E} b_{z}$ and $z\mapsto \star^{-1}\otimes b_{z}^{-1}$ are holomorphic.
Therefore, the coefficients of the symbols of $\Delta_{E,g,b_{z}}$ and $\mathcal{B}_{E,g,b_{z}}$
are holomorphic functions in $z\in O$. %for each point in $U$.
Now, the expression
$\LIM_{t\rightarrow 0}(\STr(\psi \exp(-t\Delta_{\mathcal{B},b_{z}})))$ is computed with the formula
(\ref{equation_coefficients_asymptotic_expansion}), by
integrating the complex-valued function $\STr(\psi\cdot\mathfrak{e}_{m}(\Delta_{E,g,b_{z}}))$ over $U$, and the complex-valued function
$\STr({\nabla_{\varsigma_{\mathsf{in}}}}^{k}\psi\cdot\mathfrak{e}_{m,k}(\Delta_{E,g,b_{z}},\mathcal{B}_{E,g,b_{z}}))$
over $U\cap\partial M$. Since
$\mathfrak{e}_{m}(\Delta_{E,g,b_{z}})$ are locally computable endomorphism invariants,
the value of $\STr_{x}(\psi_{x}\cdot\mathfrak{e}_{m}(\Delta_{E,g,b_{z}})_{x})$ can be computed inductively
by using explicit formulas
as a universal polynomial
in terms of (finite number of the derivatives of) the coefficients of the symbol of $\Delta_{E,g,b_{z}}$,
whenever these are given in local coordinates around at $x\in M$, see \cite[Theorem 3]{Seeleyb}, \cite[formulas (3)-(6) and Lemma 1]{Seeleya}, see also \cite[Section 2.6]{Greiner}.
In the same token, since $\mathfrak{e}_{m,k}(\Delta_{E,g,b_{z}},\mathcal{B}_{E,g,b_{z}})$
are locally computable endomorphism invariants on the boundary,
the value of $\STr_{y}(({\nabla_{\varsigma_{\mathsf{in}}}}^{k}\psi)_{y}\cdot\mathfrak{e}_{m,k}(\Delta_{E,g,b_{z}},\mathcal{B}_{E,g,b_{z}})_{y})$ is expressible,
by inductively solving certain systems of ordinary differential equations,
as a universal polynomial
in terms of (finite number of the derivatives of) the coefficients of the symbols of $\Delta_{E,g,b_{z}}$ and $\mathcal{B}_{E,g,b_{z}}$,
whenever these are given in local coordinates around at $y\in\partial M$,
see  \cite[Theorem 3]{Seeleyb}, \cite[formulas (9)-(14) and Lemma 2]{Seeleya}, see also \cite[Section 2.6]{Greiner}. %The procedure needed for the computation of these symbols is explicitly given in
Thus the mappings
$z\mapsto\STr_{x}(\mathfrak{e}_{m}(\Psi,\Delta_{z})_{x})$ and $z\mapsto\STr_{x}(\mathfrak{e}_{m,k}(\Psi,\Delta_{z},\mathcal{B}_{z})_{x})$
are holomorphic on $O$ for each $x\in U$. Finally, by Morera's Theorem,
the integral of a function depending holomorphically on a parameter $z$, also depends holomorphically on $z$, that is,
the function
$z\mapsto\LIM_{t\rightarrow 0}\left(\STr\left(\psi \exp(-t\Delta_{\mathcal{B},b_{z}})\right)\right)$ depends holomorphically
on $z\in O$.
\end{proof}

\begin{thm}
\label{Theorem infinitesimal_constant_terms_in_asymptotic_expansions_bilinear_case_relative_and absolute_boundary_conditions}
For $(M,\partial_{+}M,\partial_{-}M)$ consider the bilinear boundary value problem
$[\Delta,\mathcal{B}]^{E,g,b}_{(M,\partial_{+}M,\partial_{-}M)}$, together with
its $L^{2}$-realization $\Delta_{\mathcal{B},b}$.
If
$\phi$, $\xi$ and $\Psi$ are as in Proposition
\ref{Proposition_infinitesimal_constant_term_in_asymptotic_expansion_Hermitian_absolute_boundary_conditions}, then
%\begin{equation}
 \begin{eqnarray}
  \nonumber\scriptstyle\LIM_{t\rightarrow 0} \left(\STr\left(\phi
   \exp(-t\Delta_{\mathcal{B},b})\right)\right)&=&\scriptstyle\int_{M}\Tr(\phi)\mathbf{e}(M,g)+(-1)^{m-1}\int_{\partial_{+} M}\Tr(\phi) i^{*}_{+}\mathbf{e_{b}}(\partial M,g)\\
 %{ }\hspace{3.5cm}\phantom{=}
 &&\scriptstyle-\int_{\partial_{-} M}\Tr(\phi)i_{-}^{*}
  \mathbf{e_{b}}(\partial M,g),   \label{equation_anomalyformula_1}
\end{eqnarray}
%\end{equation}
and

 \begin{eqnarray}
 \nonumber\scriptstyle\LIM_{t\rightarrow 0} \left(\STr\left(-\Psi
 \exp(-t\Delta_{\mathcal{B},b})\right)\right)&=&\scriptstyle-2\int_{M} \left.\frac{\partial}{\partial\tau}\right|_{\tau=0
 }\mathbf{\widetilde{e}}(M,g,g+\tau g\xi)\wedge
 \omega(\nabla^{E},b)\\
&&\nonumber\scriptstyle-2\int_{\partial_{+}M} \left.\frac{\partial}{\partial\tau}\right|_{\tau=0
	}i^{*}_{+}\mathbf{\widetilde{e}_{b}}(\partial M,g,g+\tau g\xi)\wedge
	\omega(\nabla^{E},b)\\
&&\label{equation_anomalyformula_2}\scriptstyle+\mathsf{rank}(E)\int_{\partial_{+} M}\left.\frac{\partial}{\partial\tau}\right|_{\tau=0}
             i^{*}_{+}B(\partial M,g+\tau g\xi)\\
&&\nonumber\scriptstyle-2(-1)^{m}\int_{\partial_{-}M} \left.\frac{\partial}{\partial\tau}\right|_{\tau=0
	}i^{*}_{-}\mathbf{\widetilde{e}_{b}}(\partial M,g,g+\tau g\xi)\wedge
	\omega(\nabla^{E},b)\\
&&\nonumber\scriptstyle+(-1)^{m+1}\mathsf{rank}(E)\int_{\partial_{-} M}\left.\frac{\partial}{\partial\tau}\right|_{\tau=0}
	      i^{*}_{-}B(\partial M,g+\tau g\xi).
\end{eqnarray}

\end{thm}
 \begin{proof}
 By compactness of $M$, it suffices to show that each point $x\in M$ admits a neighborhood $U$ so
 that the formulas above hold for all $\phi$ with $\mathsf{supp}(\phi)\subset U$ and $\xi$ with $\mathsf{supp}(\xi)\subset U$.
 %We thus fix $x\in M$,
 For each $x\in M$, choose
 $
 b_{z}=b+z\widetilde{b}$,  $h_{s}$ %=h+s\widetilde{h}
 and $U$ as in
 Proposition \ref{Proposition_family_of_bilinear_and_Hermitian_forms}, with $\mathsf{supp}(\phi)\subset U$.
 By Proposition \ref{Proposition_family_of_bilinear_and_Hermitian_forms} (iii), we obtain
 $
 \LIM_{t\rightarrow 0 }\STr(\phi \exp(-t\Delta_{\mathcal{B},b_{s-\mathbf{i}}}
   ))=
 \LIM_{t\rightarrow 0 }\STr\left(\phi \exp(-t\Delta_{\mathcal{B},h_{s}})\right),
 $
 for all $|s|\leq 1$, for these quantities depend on the geometry over $U$ only.
 From Theorem
 \ref{Theorem infinitesimal_constant_terms_in_asymptotic_expansions_Hermitian_case_relative_and absolute_boundary_conditions}, we have
 \begin{eqnarray}
   \nonumber\scriptstyle\LIM_{t\rightarrow 0 }\STr\left(\phi \exp(-t\Delta_{\mathcal{B},b_{s-\mathbf{i}}}
   )\right)&=&\scriptstyle
\int_{M}\Tr(\phi)\mathbf{e}(M,g)
  +(-1)^{m-1}\int_{\partial_{+} M}\Tr(\phi) i^{*}_{+}\mathbf{e_{b}}(\partial M,g)\\
&&\nonumber\scriptstyle-\int_{\partial_{-} M}\Tr(\phi)
  i^{*}_{-}\mathbf{e_{b}}(\partial M,g)
\end{eqnarray}
 for all $|s|\leq 1$.
Now, since the function
$z\mapsto\LIM_{t\rightarrow 0}\STr(\phi \exp(-t\Delta_{\mathcal{B},b_{z}}))$
depends holomorphically on $z$
(see Lemma \ref{Lemma_holomorphic_dependance_of_a_m_on_a_parameter_u}), that the right hand side of the equality above is constant
in $z$, and that the domain of definition of $z$ contains an accumulation point, these formulas are extended by analytic continuation to
 \begin{eqnarray}
     \nonumber\scriptstyle\LIM_{t\rightarrow 0}\STr\left(\phi \exp(-t\Delta_{\mathcal{B},b_{z}})\right)&=&\scriptstyle
\int_{M}\Tr(\phi)\mathbf{e}(M,g)
  +(-1)^{m-1}\int_{\partial_{+} M}\Tr(\phi) i^{*}_{+}\mathbf{e_{b}}(\partial M,g)\\
&&\nonumber\scriptstyle-\int_{\partial_{-} M}\Tr(\phi)
  i^{*}_{-}\mathbf{e_{b}}(\partial M,g),
\end{eqnarray}
for all $|z|\leq\sqrt{2}$.
After setting $z=0$ we obtain the desired identity in (\ref{equation_anomalyformula_1}).
We now show (\ref{equation_anomalyformula_2}). Similarly take $\xi$ with $\mathsf{supp}(\xi)\subset U$, using Proposition \ref{Proposition_family_of_bilinear_and_Hermitian_forms} (iii), we obtain
\begin{equation}
 \label{equation_proof_equation_anomalyformula_2}
 \LIM_{t\rightarrow 0}\STr(-\Psi
 \exp(-t\Delta_{\mathcal{B},b_{s-\mathbf{i}}}
   ))=
 \LIM_{t\rightarrow 0 }\STr(-\Psi \exp(-t\Delta_{\mathcal{B},h_{s}}))
\end{equation}
for all $|s|\leq 1$, for these quantities depend on the geometry over $U$ only. Then, we apply Theorem
 \ref{Theorem infinitesimal_constant_terms_in_asymptotic_expansions_Hermitian_case_relative_and absolute_boundary_conditions}
to the right hand side of the equality in (\ref{equation_proof_equation_anomalyformula_2}) we conclude

 \begin{eqnarray}
     \nonumber\scriptstyle
 \LIM_{t\rightarrow 0}\STr\left(-\Psi
 \exp(-t\Delta_{\mathcal{B},b_{s-\mathbf{i}}}
   )\right)&=& \scriptstyle
-2\int_{M} \left.\frac{\partial}{\partial\tau}\right|_{\tau=0
 }\mathbf{\widetilde{e}}(M,g,g+\tau g\xi)\wedge
 \omega(\nabla^{E},b_{s-\mathbf{i}})\hspace{2cm}\\
&&\nonumber\scriptstyle	-2\int_{\partial_{+}M} \left.\frac{\partial}{\partial\tau}\right|_{\tau=0
	}i^{*}_{+}\mathbf{\widetilde{e}_{b}}(\partial M,g,g+\tau g\xi)\wedge
	\omega(\nabla^{E},b_{s-\mathbf{i}})\\
&&\label{equation_proof_equation_anomalyformula_2_2}\scriptstyle
            +\mathsf{rank}(E)\int_{\partial_{+} M}\left.\frac{\partial}{\partial\tau}\right|_{\tau=0}
             i^{*}_{+}B(\partial M,g+\tau g\xi)\\
&&\nonumber\scriptstyle	-2(-1)^{m}\int_{\partial_{-}M} \left.\frac{\partial}{\partial\tau}\right|_{\tau=0
	}i^{*}_{-}\mathbf{\widetilde{e}_{b}}(\partial M,g,g+\tau g\xi)\wedge
	\omega(\nabla^{E},b_{s-\mathbf{i}})\\
&&\nonumber\scriptstyle     +(-1)^{m+1}\mathsf{rank}(E)\int_{\partial_{-} M}\left.\frac{\partial}{\partial\tau}\right|_{\tau=0}
	      i^{*}_{-}B(\partial M,g+\tau g\xi),
\end{eqnarray}
 for all $|s|\leq 1$.
 Now, the function
 $z\mapsto\LIM_{t\rightarrow 0}\STr(\phi \exp(-t\Delta_{\mathcal{B},b_{z}}))$ on the left of (\ref{equation_proof_equation_anomalyformula_2_2})
 depends holomorphically on $z$ see Lemma \ref{Lemma_holomorphic_dependance_of_a_m_on_a_parameter_u}. On the other hand
 the long expression on the right hand side of the
 equality above in (\ref{equation_proof_equation_anomalyformula_2_2}) is also a holomorphic function in $z\in \C$ with $|z|\leq\sqrt{2}$, since it can be formally considered as the composition of constant functions (in $z$)
 and the function
 $
 z\mapsto\omega(\nabla^{E},b_{z})=-\frac{1}{2}\Tr(b^{-1}_{z}\nabla^{E}b_{z}),
 $
 which is holomorphic, since by Proposition \ref{Proposition_family_of_bilinear_and_Hermitian_forms} the bilinear form $b_{z}$ in (\ref{family_b_z})
 is fiberwise nondegenerate for $|z|\leq\sqrt{2}$.
 Then the identity in (\ref{equation_proof_equation_anomalyformula_2_2}) can be analytically extended to

 \begin{eqnarray}
     \nonumber\scriptstyle
 \LIM_{t\rightarrow 0}\STr\left(-\Psi
 \exp(-t\Delta_{\mathcal{B},b_{z-\mathbf{i}}}
   )\right)&=&\scriptstyle
-2\int_{M} \left.\frac{\partial}{\partial\tau}\right|_{\tau=0
 }\mathbf{\widetilde{e}}(M,g,g+\tau g\xi)\wedge
 \omega(\nabla^{E},b_{z-\mathbf{i}})\hspace{2cm}\\
&&\nonumber\scriptstyle	-2\int_{\partial_{+}M} \left.\frac{\partial}{\partial\tau}\right|_{\tau=0
	}i^{*}_{+}\mathbf{\widetilde{e}_{b}}(\partial M,g,g+\tau g\xi)\wedge
	\omega(\nabla^{E},b_{z-\mathbf{i}})\\
&&\label{equation_proof_equation_anomalyformula_2_3}\scriptstyle +\mathsf{rank}(E)\int_{\partial_{+} M}\left.\frac{\partial}{\partial\tau}\right|_{\tau=0}
             i^{*}_{+}B(\partial M,g+\tau g\xi)\\
&&\nonumber\scriptstyle	-2(-1)^{m}\int_{\partial_{-}M} \left.\frac{\partial}{\partial\tau}\right|_{\tau=0
	}i^{*}_{-}\mathbf{\widetilde{e}_{b}}(\partial M,g,g+\tau g\xi)\wedge
	\omega(\nabla^{E},b_{z-\mathbf{i}})\\
&&\nonumber\scriptstyle	      +(-1)^{m+1}\mathsf{rank}(E)\int_{\partial_{-} M}\left.\frac{\partial}{\partial\tau}\right|_{\tau=0}
	      i^{*}_{-}B(\partial M,g+\tau g\xi),
\end{eqnarray}
for $z\in\C$ with $|z-\mathbf{i}|\leq\sqrt{2}$. Finally (\ref{equation_anomalyformula_2}) follows from setting $z=i$ into (\ref{equation_proof_equation_anomalyformula_2_3}) and then $b_{0}=b$ follows from
(\ref{family_b_z}).
 \end{proof}

\section{Complex-valued analytic torsion on compact Bordisms}
\label{Subsection_Complex analytic Ray--Singer torsion}
Let $(M,\partial_{+}M,\partial_{-}M)$ be a Riemannian bordism
and $E$ be complex flat vector bundle over $M$ endowed with a nondegenerate symmetric bilinear form.
Consider $\Delta_{\mathcal{B}}$ the $L^{2}$-realization of the bilinear
Laplacian acting on $E$-valued smooth forms satisfying absolute boundary conditions on $\partial_{+}M$ and
relative ones on $\partial_{-}M$.

If $\Omega_{\Delta_{\mathcal{B}}}(0)$ is the $0$-generalized eigenspace of $\Delta_{\mathcal{B}}$, consider the
restriction of $\beta_{g,b}$ to
$\Omega_{\Delta_{\mathcal{B}}}(0)$; this is a
non degenerate symmetric bilinear form in view of Proposition \ref{Proposition_orthogonal_decomposition_smooth_forms}.
By \cite[Lemma 3.3]{Burghelea-Haller} we obtain
a nondegenerate bilinear form on
$\det H\left({\Omega_{\Delta_{\mathcal{B}}}(0)}\right)$, which in turn,
by Proposition \ref{Proposition_Absolute_relative_cohomology}, induces a bilinear form on
$\det(H(M,\partial_{-}M;E))$, which we
denote by $\tau(0)_{E,g,b}$.
Let us denote by
$$
\Delta_{\mathcal{B},q}^{\mathsf{c}}:=\left.\Delta_{\mathcal{B}}\right|_{\left.{\Omega^{q}_{\Delta_{\mathcal{B}}}(M;E)(0)^{\mathsf{c}}}\right.|_{\mathcal{B}}}
$$
the restriction of $\Delta_{\mathcal{B}}$
to $\left.{\Omega^{q}_{\Delta_{\mathcal{B}}}(M;E)(0)^{\mathsf{c}}}\right.|_{\mathcal{B}}$, i.e., the space
of smooth differential forms of degree $q$
which are not in  $\Omega_{\Delta_{\mathcal{B}}}(M;E)(0)$ but satisfy boundary conditions.
Lemma \ref{Proposition_spectrum_Delta_L_2_invariance_of_domain_of_delta_contained_in_domain_of_d}
permits us to choose a non-zero Agmon angle avoiding the
spectrum of
$\Delta_{\mathcal{B},q}^{\mathsf{c}}$ so that
complex powers of the bilinear Laplacian can be defined.
Then, the function
$
s\mapsto(\Delta_{\mathcal{B},q}^{\mathsf{c}})^{-s}
$
associates to each $s\in\C$, with $\mathsf{Re}(s)>\mathsf{dim}(M)/2$, an operator of
Trace class and it extends to a meromorphic function on the complex plane which is holomorphic at $0$,
see \cite{Greiner}, \cite{Seeley}, \cite{Seeleya} and\cite{Seeleyb} or more generally, for pseudo-differential boundary value problems,
see \cite[Chapter 4]{Grubb}. The
$\zeta$-regularized determinant of $\Delta_{\mathcal{B},q}$ is defined as
$$
{\det}^{\prime}\left(\Delta_{\mathcal{B},q}\right):=
\exp(-\left.\frac{\partial}{\partial s}\right|_{s=0}\Tr((\Delta_{\mathcal{B},q}^{\mathsf{c}})^{-s})).
$$
From Lemma \ref{Proposition_spectrum_Delta_L_2_invariance_of_domain_of_delta_contained_in_domain_of_d}
this determinant does not depend on the choice of the Agmon's angle.
By using \cite[Lemma 3.3]{Burghelea-Haller}, the complex-valued Ray--Singer torsion on the bordism
$(M,\partial_{+}M,\partial_{-}M)$ is defined as the bilinear form on the determinant line $\det H(M,\partial_{-}M;E)$ given by
$$
  \tau_{E,g,b}:=\tau(0)_{E,g,b}\prod_{q}\left({\det}^{\prime}\left(\Delta_{\mathcal{B},q}\right)\right)^{(-1)^{q}q}.
$$
The following generalizes the formulas
obtained in \cite{Burghelea-Haller} in the case
without boundary and they are based on the corresponding ones for
the Ray--Singer metric in \cite{Bruening-Ma}.
They also coincide with
the ones obtained by Su in odd dimensions, but they do not
require that the
smooth variations of $g$ and
$b$ are supported on a compactly supported in the interior of $M$, see \cite{Su1}.
 \begin{thm}(\textbf{Anomaly formulas})
\label{Theorem_Anomaly_formulas_complex_Analytic_Ray_Singer_torsion}
  Let $(M,\partial_{+}M,\partial_{-}M)$ be a compact Riemannian bordism
  and $E$ be complex flat vector bundle over $M$.
  Consider $g_{u}$ a smooth one-parameter family of Riemannian metrics on $M$
  and  $b_{u}$ a smooth one-parameter family of a fiber wise nondegenerate symmetric
  bilinear forms on $E$ and denote by $\dot{g}_{t}$ and $\dot{b}_{t}$
  their corresponding infinitesimal variations. Let $\tau_{E,g_{u},b_{u}}$ the associated
  family of complex valued analytic torsions.
  Then, we have the following logarithmic derivative
\begin{eqnarray}
  \nonumber\left.\frac{\partial}{\partial w}\right|_{u}
  \left(\frac{\tau_{E,g_{w},b_{w}}}{\tau_{E,g_{u},b_{u}}}\right)^{2}
  =\mathbf{E}(b_{u},g_{u})+ \mathbf{\tilde{E}}(b_{u},g_{u})+ \mathbf{B}(g_{u}),
  \end{eqnarray}
 where $\omega(\nabla^{E},b):=-\frac{1}{2}\Tr(\left.b\right.^{-1}\nabla^{E}b)$
  is the Kamber--Tondeur form, see \cite[Section 2.4]{Burghelea-Haller} and
\begin{eqnarray}
	      \nonumber\scriptstyle\mathbf{E}(b_{u},g_{u})&:=&
	      \scriptstyle\int_{M}\Tr(b^{-1}_{u}\dot{b}_{u})\mathbf{e}(M,g)
	    +(-1)^{m-1}\int_{\partial_{+} M}\Tr(b^{-1}_{u}\dot{b}_{u})
	      \mathbf{e_{b}}(\partial M,g_{u})\\
	    &&\nonumber\scriptstyle-\int_{\partial_{-} M}\Tr({b^{\prime}}^{-1}_{u}{\dot{b}}^{\prime}_{u}))
	    \mathbf{e_{b}}(\partial M,g_{u}),
  \end{eqnarray}
\begin{eqnarray}
	     \nonumber \scriptstyle\mathbf{\tilde{E}}(b_{u},g_{u})&:=&\scriptstyle-2\int_{M} \left.\frac{\partial}{\partial t}\right|_{t=0
	      }\mathbf{\widetilde{e}}(M,g_{u},g_{u}+t \dot{g}_{u})\wedge
		  \omega(\nabla^{E},b_{u})\\
	    &&\nonumber\scriptstyle-2\int_{\partial_{+}M} \left.\frac{\partial}{\partial t}\right|_{t=0
	    }i^{*}_{+}\mathbf{\widetilde{e}_{b}}(\partial M,g_{u},g_{u}+t \dot{g}_{u})\wedge
	  \omega(\nabla^{E},b_{u})\\
    &&\nonumber\scriptstyle-2(-1)^{m}\int_{\partial_{-}M} \left.\frac{\partial}{\partial t}\right|_{t=0
	  }i^{*}_{-}\mathbf{\widetilde{e}_{b}}(\partial M,g_{u},g_{u}+t \dot{g}_{u})\wedge
	  \omega(\nabla^{E},b_{u}),
  \end{eqnarray}
\begin{eqnarray}
	  \nonumber\scriptstyle\mathbf{B}(g_{u})&:=&\scriptstyle\mathsf{rank}(E)\int_{\partial_{+} M}\left.\frac{\partial}{\partial t}\right|_{t=0}
	  i_{+}^{*}B(\partial M,g_{u}+ t \dot{g}_{u})\\
	  &&\nonumber\scriptstyle+(-1)^{m+1}\mathsf{rank}(E)\int_{\partial_{-} M}\left.\frac{\partial}{\partial t}\right|_{t=0}
	i_{-}^{*}B(\partial M,g_{u}+t \dot{g}_{u}),
  \end{eqnarray}
  \end{thm}
\begin{proof}
 The method described in \cite[Section 6]{Burghelea-Haller} leading
 to the infinitesimal variation of the torsion in the closed situation also holds
 in the situation with boundary; this was also used in \cite{Su1}.
 In particular, by \cite[formula (54)]{Burghelea-Haller},
 the problem of computing this infinitesimal variation boils down to computing
 $\LIM_{t\rightarrow 0}(\STr(\phi \exp(-t\Delta_{\mathcal{B}})))$ and $\LIM_{t\rightarrow 0}(\STr(-\Psi \exp(-t\Delta_{\mathcal{B}})))$
 associated to $\Delta_{\mathcal{B}}$ with
 $\phi=b_{u}^{-1}\dot{b}_{u}$ and  $\xi=g_{u}^{-1}\dot{g}_{u}$ respectively given
 by (\ref{equation_anomalyformula_1}) and (\ref{equation_anomalyformula_2}) in
 Theorem \ref{Theorem infinitesimal_constant_terms_in_asymptotic_expansions_bilinear_case_relative_and absolute_boundary_conditions}.
\end{proof}

  \section{Appendix}
  \label{section_Appendix}
  The material in this section, entirely contained in \cite{Bruening-Ma}, summarizes the background needed to understand the characteristic forms
  appearing in the anomaly formulas in
  Sections \ref{section_Heat_asymptotics and anomaly formulas} and \ref{Subsection_Complex analytic Ray--Singer torsion}.
  \subsection{The Berezin integral and Pfaffian}
  \label{section_The Berezin integral and Pfaffian}
  For $A$ and $B$ two $\Z_{2}$ graded unital algebras,
  $A\widehat{\otimes}B$ denotes their $\Z_{2}$-\textit{graded tensor product}. We write
  $A:=A\widehat{\otimes} I$, $\widehat{B}:=I\widehat{\otimes} B$ and $\wedge:=\widehat{\otimes},$ so that
  $A\wedge\widehat{B}=A\widehat{\otimes} B.$

  For $W$ and $V$ finite dimensional vector
  spaces of dimension $n$ and $l$ respectively, where $W$ is endowed with a Hermitian product $\langle\cdot,\cdot\rangle$
  and $V^{\prime}$ the dual of $V$, the \textit{Berezin integral} $\int^{B}:\Lambda V^{\prime}\wedge\widehat{\Lambda(W^{\prime})}\rightarrow\Lambda V^{\prime}\otimes\Theta_{W}$ 
  associates to each element $\alpha\wedge\widehat{\beta}$ in the $\Z_{2}$-graded tensor product
  $\Lambda V^{\prime}\wedge\widehat{\Lambda(W^{\prime})}$ the element $C_{B}\beta_{g,b}(w_{1},\ldots,w_{n})$ in $\Lambda V^{\prime}\otimes\Theta_{W}$, 
  where $\{w_{i}\}_{i=1}^{n}$ is an orthonormal basis of $W$, $\Theta_{W}$ is the orientation bundle of $W$ and the constant
  $C_{B}:=(-1)^{n(n+1)/2}\pi^{-n/2}$. Now, each antisymmetric endomorphism $K$ of $W$ can be identified with
  the unique element $\mathbf{K}:=\langle\cdot,K\cdot\rangle$ in $\Lambda(W^{\prime})$ given by
  $$
  \mathbf{K}:=\frac{1}{2}\sum_{1\leq i,j\leq n}\langle w_{i},K w_{j}\rangle\widehat{w^{i}}\wedge\widehat{w^{j}},
  $$
  where $\{w^{i}\}_{i=1}^{n}$ is the corresponding dual basis in $W^{\prime}$.
  Then, $\mathbf{Pf}\left(\mathbf{K}/2\pi\right)$, the \textit{Pfaffian} of $\mathbf{K}/2\pi$, is defined by
  $$
  \mathbf{Pf}\left(\mathbf{K}/2\pi\right):=\int^{B}\exp(\mathbf{K}/2\pi).
  $$
  Remark that $\mathbf{Pf}\left(\mathbf{K}/2\pi\right)=0$, if $n$ is odd.
  By standard fiberwise considerations the map $\mathbf{Pf}$ is extended for
  vector bundles over $M$.
  \subsection{Certain characteristic forms on the boundary}
  We denote by $g:=g^{TM}$ (resp.  $g^{\partial}:=g^{T\partial M}$) the Riemannian metric on $TM$
  (resp. on $T\partial M$ and induced by $g$), by $\nabla$ (resp. $\nabla^{\partial}$) the
  corresponding Levi-Civita connection and by $\mathsf{R}^{{TM}}$ (resp. $\mathsf{R}^{{T\partial M}}$) its curvature.
  Let $\{e_{i}\}_{i=1}^{m}$ be an orthonormal frame of $TM$ with the property that near the
  boundary, $e_{m}=\varsigma_{\mathsf{in}}$, i.e., the inwards pointing geodesic unit normal vector field on the boundary.
  The corresponding induced orthonormal local frame on $T\partial M$ will be denoted by $\{e_{\alpha}\}_{\alpha=1}^{m-1}$. As usual, the metric is used to fix
  $\{e^{i}\}_{i=1}^{m}$ (resp. $\{e^{\alpha}\}_{\alpha=1}^{m-1}$) the corresponding dual frame of $T^{*}M$ (resp. $T^{*}\partial M$).

  With the notation in Section \ref{section_The Berezin integral and Pfaffian}, a smooth section $w$ of $\Lambda T^{*}M$ 
  is identified with the section $w\widehat{\otimes}1$ of $\Lambda T^{*}M\widehat{\otimes}\widehat{\Lambda T^{*}M}$, whereas
  $\widehat{w}$ is in one-to-one correspondance with the section $1\widehat{\otimes}\widehat{w}$ of $\Lambda T^{*}M\widehat{\otimes}\widehat{\Lambda T^{*}M}$.
  Here one considers the Berezin integrals
  $
    \int^{B_{M}}:\Gamma(M;\Lambda T^{*}M\wedge\widehat{\Lambda T^{*}M})\rightarrow\Gamma(M;\Lambda T^{*}M\otimes\Theta_{M})
  $
  and
  $
    \int^{B_{\partial M}}:\Gamma(\partial M;\Lambda T^{*}\partial M\wedge\widehat{\Lambda(T^{*}\partial M)})\rightarrow\Gamma(\partial M;\Lambda T^{*}\partial M\otimes\Theta_{\partial M})
  $
  which can be compared under the taken convention for the induced orientation bundle on the boundary
  discussed in Section \ref{Section_Complex valued analytic torsion on Manifolds with Boundary}.

  The curvature
 $\mathsf{R}^{{TM}}$ associated to $\nabla$, considered as a smooth section of
 $\Lambda^{2}(T^{*}M)\wedge\widehat{\Lambda^{2}(T^{*}M})\rightarrow M,$
 can be expanded in terms of the frame above as 
$$
\begin{array}{rcll}
\scriptstyle\mathbf{R}^{{TM}}&\scriptstyle:=&\scriptstyle\frac{1}{2}\sum_{1\leq k,l\leq m}g^{{TM}}\left(e_{k},\mathsf{R}^{{TM}}e_{l}\right)\widehat{e^{k}}\wedge\widehat{e^{l}}
&\scriptstyle\in\quad\Gamma({M};\Lambda^{2}(T^{*}{M})\wedge\widehat{\Lambda^{2}({T^{*}M})})\\
\end{array}
$$
In the same way, one first sets
\begin{equation}
\label{equation_definition_of_S_both_boundaries}
\begin{array}{l}
\scriptstyle i^{*}\mathbf{R}^{{TM}}\scriptstyle:=\scriptstyle\frac{1}{2}\sum\limits_{1\leq k,l\leq m}
g^{{TM}}(e_{k},i^{*}\mathsf{R}^{{TM}}e_{l})\widehat{e^{k}}\wedge\widehat{e^{l}}
\quad\in\quad\Gamma({\partial M};\Lambda^{2}(T^{*}{\partial M})\wedge\widehat{\Lambda^{2}({T^{*}M})}),\\

\scriptstyle\left.\mathbf{R}^{{TM}}\right|_{\partial M}\scriptstyle:=\scriptstyle\frac{1}{2}\sum\limits_{1\leq \alpha,\beta\leq m-1}
g^{{TM}}(e_{\alpha},i^{*}\mathsf{R}^{{TM}}e_{\beta})\widehat{e^{\alpha}}\wedge\widehat{e^{\beta}}
\scriptstyle\quad\in\quad\Gamma({\partial M};\Lambda^{2}(T^{*}{\partial M})\wedge\widehat{\Lambda^{2}({T^{*}(\partial M)}))}),\\

\scriptstyle\mathbf{R}^{{T \partial M}}
\scriptstyle:=\scriptstyle\frac{1}{2}\sum\limits_{1\leq \alpha,\beta\leq m-1}
g^{{T \partial M}}(e_{\alpha},i^{*}\mathsf{R}^{{TM}}e_{\beta})\widehat{e^{\alpha}}\wedge\widehat{e^{\beta}}
\scriptstyle\quad\in\quad\Gamma({\partial M};\Lambda^{2}(T^{*}{\partial M})\wedge\widehat{\Lambda^{2}({T^{*}(\partial M)}))}),\\

\scriptstyle\mathbf{S}\scriptstyle:=
\scriptstyle\frac{1}{2}\sum\limits_{\beta=1}^{m-1}
\left(\sum\limits_{\alpha=1}^{m-1}g^{{TM}}(\nabla^{{TM}}_{e_{\alpha}}\varsigma_{\mathsf{in}},e_{\beta})e^{\alpha}\right)\wedge\widehat{e^{\beta}}
\scriptstyle\quad\in\quad\Gamma({\partial M};T^{*}{\partial M}\wedge\widehat{\Lambda^{1}({T^{*}(\partial M)})})\\
\end{array}
\end{equation}
%By \cite[(1.16)]{Bruening-Ma}, the relation
%$\mathbf{R}^{{T \partial M}}=\left.\mathbf{R}^{{TM}}\right|_{\partial M}-2\mathbf{S}$ holds.
to define
\begin{equation}
\label{label_in_article_differential_forms_boundary}
\begin{array}{rcl}
\mathbf{e}({M},\nabla^{{T M}})&:=&\int^{B_{M}}\exp\left(-\frac{1}{2}\mathbf{R}^{{T M}}\right),\\

\mathbf{e}({\partial M},\nabla^{{T \partial M}})&:=&\int^{B_{\partial M}}\exp\left(-\frac{1}{2}
\mathbf{R}^{{T \partial M}}\right),\\

\mathbf{e}_{\mathbf{b}}({\partial M}, \nabla^{{T M}})&:=&
(-1)^{m-1}\int^{B_{\partial M}}\exp\left(-\frac{1}{2}(\mathbf{R}^{{T M}}|_{\partial  M})\right)\sum_{k=0}^{\infty}\frac{\mathbf{S}^{k}}{2\Gamma(\frac{k}{2}+1)},\\

B({\partial M}, \nabla^{{T M}})  &:=&
-\int^{1}_{0}\frac{du}{u}\int^{B_{\partial M}}\exp\left(-\frac{1}{2}\mathbf{R}^{{T \partial M}}-u^{2}\mathbf{S}^{2}\right)\sum_{k=1}^{\infty}\frac{\left(u\mathbf{S}\right)^{k}}{2\Gamma(\frac{k}{2}+1)}.

\end{array}
\end{equation}
\subsection{Secondary characteristic forms}
Now, given $\{g_{s}:=g_{s}^{TM}\}_{s\in\R}$ (resp. $\{g_{s}^{\partial}:=g_{s}^{T\partial M}\}_{s\in\R}$) a smooth family
of Riemannian metrics on $TM$ (resp. the induced family of metrics on $T\partial M$), we sketch the construction given in \cite{Bruening-Ma}
for the (secondary) \textit{Chern--Simons forms}  $\widetilde{\mathbf{e}}\left(M,g_{0},g_{s}\right)\in\Omega^{m-1}(M,\Theta_{M})$
and $\left.\widetilde{\mathbf{e}}_{\mathbf{b}}\right.\left( \partial M,g_{0},g_{s}\right)\in\Omega^{m-2}(\partial M,\Theta_{M})$
(see also \cite[(4.53)]{Bismut-Zhang}).

Let
$\nabla_{s}:=\nabla_{g_{s}}^{TM}$ and $\mathsf{R}_{s}:={\mathsf{R}}^{TM}_{g_{s}}$ (resp.
$\nabla_{s}^{\partial}:=\nabla_{g^{\partial}_{s}}^{T\partial M}$ and $\mathsf{R}^{\partial}_{s}:=\mathsf{R}^{T\partial M}_{g^{\partial}_{s}}$)
be the Levi-Civit\`a connections and curvatures
on $TM$ (resp. on $T\partial M$) associated to the metrics $g_{s}$ (resp. $g^{\partial}_{s}$).
Consider the \textit{deformation spaces} $\widetilde{M}:=M\times\R$ (resp. $\widetilde{\partial M}:=\partial M\times \R$) with
$\pi_{\widetilde{M}}:\widetilde{M}\rightarrow \R\text{ and }\mathbf{p}_{M}:\widetilde{M}\rightarrow M,$ its canonical projections (resp.
$\pi_{\widetilde{\partial M}}:\widetilde{\partial M}\rightarrow \R\text{ and }\mathbf{p}_{\partial M}:
\widetilde{\partial M}\rightarrow \partial M$).
%and denote by
If $\widetilde{i}:=i\times\mathbf{id_\R}:\widetilde{\partial M}\rightarrow \widetilde{M}$ is the natural embedding
induced by $i:\partial M\rightarrow M$,
%$$
%\pi_{\widetilde{M}}:\widetilde{M}\rightarrow \R,\quad \mathbf{p}_{M}:\widetilde{M}\rightarrow M \quad\text{and}\quad
%\pi_{\widetilde{\partial M}}:\widetilde{\partial M}\rightarrow \R,\quad \mathbf{p}_{\partial M}:
%\widetilde{\partial M}\rightarrow \partial M$$
% are respectively the canonical projections,
then
$\pi_{\widetilde{\partial M}}=\pi_{\widetilde{M}}\circ \widetilde{i}$.
The \textit{vertical bundle} of the fibration $\pi_{\widetilde{M}}:\widetilde{M}\rightarrow \R$ (resp. $\pi_{\widetilde{\partial M}}:\widetilde{\partial M}\rightarrow\R$)
is given as the pull-back of the tangent bundle $TM\rightarrow M$ along  $\mathbf{p}_{M}:\widetilde{M}\rightarrow M$
(resp. the pull-back of $T\partial M\rightarrow \partial M$ along  $\mathbf{p}_{\partial M}:\widetilde{\partial M}\rightarrow \partial M$), i.e.,
\begin{equation}
 \label{equation_mathcal_T_M}
 \mathcal{T M}:=\mathbf{p}_{M}^{*}TM\rightarrow \widetilde{M},\quad(\text{resp. }\mathcal{T\partial M}:=\mathbf{p}_{\partial M}^{*}T\partial M\rightarrow \widetilde{\partial M})
\end{equation}
and it is considered as a subbundle of $T\widetilde{M}$ (resp. $T\widetilde{\partial M}$).
The bundle $\mathcal{T M}$ (resp. $\mathcal{T\partial M}$) in (\ref{equation_mathcal_T_M}) is naturally endowed with a Riemannian metric $g^{\mathcal{T M}}$ which
coincides with $g_{s}$ (resp. $g^{\partial}_{s}$) at $M\times\{s\}$ (resp. $\partial M\times\{s\}$), and for which
there exists a unique natural metric connection $\nabla^{\mathcal{T M}}$ (resp $\nabla^{\mathcal{T \partial M}}$) and the
corresponding curvature tensor is denoted by $\mathsf{R}^{\mathcal{T M}}$ (resp $\mathsf{R}^{\mathcal{T \partial M}}$). For more details,
see \cite[Section 1.5, (1.44) and Definition 1.1]{Bruening-Ma}, and also \cite[(4.50) and (4.50)]{Bismut-Zhang}).
Near the boundary, consider orthonormal frames of $\mathcal{TM}$ such that
$e_{m}(y, s) = \varsigma_{\mathsf{in}}$ for each $y\in\partial M$ with respect to the metric $g_{s}$. Finally,
by using the
formalism described above associated to $\mathsf{R}^{\mathcal{T M}}$ and $\mathsf{R}^{\mathcal{T \partial M}}$ to define (\ref{label_in_article_differential_forms_boundary}),
if $\mathbf{incl}_{s}: M \rightarrow  \widetilde{M}$
is the inclusion map given by $\mathbf{incl}_{s}(x)=(x,s)$ for $x_{0}\in M$ and $s\in \R$, then, one defines
\begin{equation}
  \label{definition_of_Chern_Simons_secondary_classes_1}
  \begin{array}{rcl}
  \widetilde{\mathbf{e}}\left(M,g_{0},g_{\tau}\right)&:=&
  \int_{0}^{\tau}\mathbf{incl}^{*}_{s}
  \left(\iota\left(\frac{\partial}{\partial s}\right)
  \mathbf{e}(\widetilde{M},\nabla^{\mathcal{T M}})\right)ds\\
  \left.\widetilde{\mathbf{e}}_{\mathbf{b}}\right.\left( \partial M,g_{0},g_{\tau}\right)&:=&
  \int_{0}^{\tau}\mathbf{incl}^{*}_{s}
  \left(\iota\left(\frac{\partial}{\partial s}\right)
  \left.\mathbf{e}_{\mathbf{b}}\right.(\widetilde{\partial M}, \nabla^{\mathcal{T M}})\right) ds,\\
  \end{array}
\end{equation}
where $\iota(X)$ indicates the contraction with respect to the vector field $X$.

\end{document}